\documentclass[a4paper,12pt]{article}
\usepackage[warn]{mathtext}
\usepackage{cmap}
\usepackage[T2A]{fontenc}
\usepackage[cp1251]{inputenc}
\usepackage[russian]{babel}

\usepackage{amsmath}
\usepackage{wasysym}
\usepackage{amssymb,amsthm}

\usepackage[dvips]{graphicx}

\theoremstyle{plain}
\newtheorem{theorem}{Теорема}[section]
\newtheorem{corollary}[theorem]{Следствие}
\newtheorem{lemma}[theorem]{Лемма}
\newtheorem{proposition}[theorem]{Предложение}

\newtheorem{remark}[theorem]{Замечание}

\DeclareMathOperator{\Evol}{Evol}

\theoremstyle{definition}
\newtheorem{definition}[theorem]{Определение}

\begin{document}

\title{Периодичность схем Рози и подстановочные системы}

\author{А.~Я.~Белов, И.~Митрофанов}

\maketitle

\begin{center}
\bigskip
Московский Государственный Университет им. М.~В.~Ломоносова.\\
Московский институт открытого образования.
\end{center}

\begin{abstract}
В работе вводится понятие {\it последовательности схем Рози} для
бесконечного вправо рекуррентного непериодичного слова. Схемы Рози
близки по своим свойствам графам Рози. Для непериодического
морфического сверхслова, то есть образа при некотором морфизме от
слова, полученного итерациями одной подстановки, доказана
периодичность последовательности схем Рози с некоторого момента.
\end{abstract}

{\bf Ключевые слова:} комбинаторика слов, проблемы Берсайдовского
типа, периодичность.

\begin{abstract}
In the paper the notion of  {\em Rauzy scheme} is introduced. From
Rauzy graph Rauzy Scheme can be obtaining by uniting sequence of
vertices of ingoing and outgoing degree 1 by arches. This notion
is a tool to describe Rauzy graph behavior. For morphic superword
we prove periodicity of Rauzy schemes. This is generalization of
fact that quadratic irrationals have periodic chain fractions.
\end{abstract}


\section{Введение.}
Комбинаторика слов находит свое применение в самых разных разделах
математики. Например, в алгебре при изучении базисов и нормальных
форм, в алгебраической топологии, в символической динамике. В \cite{PS} рассматриваются логические аспекты комбинаторики слов. Ряд
проблем, относящихся к комбинаторике слов находится на стыке
алгебры и теории динамических систем. Многие проблемы
комбинаторики слов представляют самостоятельный интерес.

Методы символической динамике играют существенную роль в изучении
комбинаторных свойств слов, задачах теории чисел и теории
динамических систем.  Пусть $M$ --- компактное топологическое
пространство, $U\subset M$ --- его открытое подмножество,
 $f\colon M \to M$ --- гомеоморфизм компакта в себя и $x_0\in M$ --- начальная точка.

 По последовательности итераций можно построить бесконечное слово $W$ над бинарным
алфавитом $\{a,b\}$:
$$
w_n=\begin{cases}
a,&\text{если $f^n(x_0)\in U$;}\\
b,&\text{если $f^n(x_0)\notin U,$}
\end{cases}
$$
которое называется эволюцией точки $x$. Символическая динамика исследует взаи-
мосвязь свойств динамической системы $(M,f)$ и комбинаторных свойств слова $W_n$.

Для слов над алфавитом, состоящим из большего числа символов нужно рассмотреть
несколько характеристических множеств: $U_1,U_2,\dots,U_n$.

Под {\it прямой задачей} символической динамики понимается изучение комбинатор-
ных свойств слов, порожденных данной динамической системой, {\it обратная задача}
символической динамики изучает свойства динамической системы, то есть свойства
компакта $M$ и преобразования $f$ по комбинаторным свойствам слова $W$.

Можно рассуждать и в обратном направлении.
Пусть $W = \{w_n\}$ -- бесконечное слово,
$\tau (\{w_n\})=\{w_{n+1}\}$ ---  оператор сдвига.
 Рассмотрим $X\subseteq A^* $ --- замыкание траектории слова относительно метрики Хэмминга.
 Прямые задачи символической динамики связаны с получением информации о динамической системе
 $(X,\tau)$ по информации о слове $W$.

 Говоря о соответствии комбинаторных свойств слов и топологической динамики,
следует указать на два обстоятельства.

Известно, что если слово $W$ {\it равномерно-рекуррентно (р.р)},
то полученная динамиче- ская система минимальна, то есть не
содержит нетривиальных замкнутых подсистем.

Свойство {\it единственности инвариантной меры} переводится на комбинаторный
язык так. Пусть $u$ есть подслово равномерно-рекуррентного сверхслова $W$. Предпо-
ложим, что для любого подслова $u \sqsubset  W$ верхняя и нижняя плотности, с которыми
оно встречается в слове $W$, совпадают. Тогда соответствующая инвариантная мера
единственная.

Можно сформулировать комбинаторные условия на то, что динамическая систе-
ма является сдвигом тора. В частности, это означает, что у нее дискретный спектр.
Определим {\it функцию рассогласования} $\rho (U; V)$  как верхнюю плотность множества по-
зиций в словах $U$ и $V$, в которых стоят разные символы.

Итак, пусть $W$ сверхслово, полученное сдвигом тора, $T$ --- оператор сдвига. То-
гда функции рассогласования между сдвигами W удовлетворяют следующим свой-
ствам.

\begin{enumerate}
\item Существует последовательность $\{n_i\}^{\infty}_{i=1}$ такая,
что $\rho(T^{n_i}(W),W)\longrightarrow 0.$
\item Существуют сколь угодно большие пары взаимно простых чисел $(n_i,n_j)$,
где $n_i,n_j\longrightarrow \infty$ из
этой последовательности.
\end{enumerate}

Для уточнения абстрактной постановки задачи исследования
соответствия комбинаторных  и топологических свойств слов очень
важна модельная ситуация, от которой можно отталкиваться при
дальнейших исследованиях. Проблематика, связанная с построением
слов Штурма очень важна в комбинаторике слов.

Задачи, прямые и обратные, связанные с преобразованием поворота
окружности приводят к классу слов, называемых {\it словами
Штурма}. Известно, что если $T(n)<n+1$ при некотором $n$, то
сверхслово $W$ периодично. Слова с предельной функцией роста
$T(n)=n+1$ образуют класс так называемых {\it слов Штурма
(Sturmian words)}, другое название -- {\it слова Бетти (Beaty
words)}, которые впервые были упомянуты в работе
\label{HedlundMorse} \cite{MorseHedlund}  Классическая теория слов
Штурма описана в обзорах \label{BerstelSeebold} \cite{BS},
\cite{AdL}. Последние продвижения в теории слов Штурма описаны в
обзоре \label{BerstelRecent} \cite{B}  А.~Т.~Колотов \cite{Kol}
\label{Kol} пришел к изучению таких слов из чистой алгебры.

Известна классическая

\begin{theorem}[Теорема эквивалентности (\cite{MorseHedlund},\cite{L1}).]
Пусть $W$ -- бесконечное рекуррентное слово над бинарным алфавитом
$A=\{a,b\}$. Следующие условия ``почти'' эквивалентны:

\begin{enumerate}
\item Слово $W$ является словом Штурма, то есть количество
различных подслов длины $n$ слова $W$ равно $T_n(W)=n+1$ для
любого $n\geq 1$.

\item Слово не периодично и является {\em сбалансированным}, то
есть для любых двух подслов $u,v\subset W$ одинаковой длины
выполняется неравенство $||v|_a-|u|_a|\leq 1$, где $|w|_a$
обозначает количество вхождений символа $a$ в слово $w$.

\item Слово $W=(w_n)$ является {\em механическим словом} с
иррациональным $\alpha$, то есть существуют такое иррациональное
$\alpha$, $x_0 \in [0,1]$ и интервал $U\subset \mathbb{S}^1$,
$|U|=\alpha$ такие, что выполняется условие:
$$
w_n=\left\{
\begin{array}{rcl}
   a,&{T_{\alpha}}^n(x_0)\in U\\
   b,&{T_{\alpha}}^n(x_0)\not\in U\\
\end{array}\right.
$$

\item Слово $W$ получается путем предельного перехода
последовательности слов, каждое из которых получается из
предыдущего путем подстановки вида $a^kb\to b, a^{k+1}b\to a$ либо
подстановки вида $b^ka\to a, b^{k+1}a\to b$.

Показатель $k$ зависит от шага. Если эти показатели $k_i$
периодически повторяются, то $\alpha$ есть квадратичная
иррациональность.

\item Сверхслово $W$ р.р. и имеет последовательность графов Рози с
одной входящей и одной исходящей развилкой. \footnote{Эта
характеризация была получена в последующем.}

\end{enumerate}
\end{theorem}

Понятие ``почти'' имеет следующий смысл: имеется счетное множество
последовательностей принадлежащих одному классу но не
принадлежащим другому и все такие исключительные
последовательности описаны. Например, при $\alpha\in \mathbb{Q}$
механические слова не принадлежат первому классу.

Во-первых, это рассмотрение {\it сбалансированных} слов над
произвольным алфавитом, а также {\it $m$-сбалансированных} слов.
Сбалансированные непериодические слова над $n$-буквенным алфавитом
изучены в работе \label{Graham} \cite{GR}, см. также \cite{H}. В
работах \cite{BC},\cite{C1} получена конструкция динамической
системы, порождающей произвольное непериодическое сбалансированное
слово.

Описание периодических сбалансированных слов связано с {\it
гипотезой Френкеля (Fraenkel's conjecture)}, утверждающей, что все
сбалансированные периодические слова над алфавитом
$A=\{a_1,\ldots, a_n\}$ из $n$ символов с попарно разными плотностями вхождения имеют вид
$$
W=(U_n)^{\infty},
$$
где $U_n$ задается рекуррентно:
$$
U_n=(U_{n-1}a_nU_{n-1}), \ U_3=a_1a_2a_1a_3a_1a_2a_1.\\
$$
Для трех буквенного алфавита гипотеза была доказана Р.~Тайдеманом
(\cite{T1,T2}). В настоящий момент гипотеза доказана для алфавитов
состоящих не более чем из $7$ символов.

Во-вторых, обобщение слов Штурма может быть получено посредством
изучения {\it функции сложности} или {\it функции роста}. {\it
Функция сложности $T_W(n)$} -- это количество различных подслов
длины $n$ слова $W$. Для слов Штурма выполняется соотношение
$T_W(n+1)-T_W(n)=1$ для всех $n\geq 1$. Поэтому естественными
обобщениями слов Штурма являются слова  над конечном алфавитом для
которых выполняется соотношение $T_W(n+1)-T_W(n)=1$ для всех
$n\geq k$, где $k$ -- некоторое натуральное число. Равносильное
условие: $T(n)=n+K$, начиная с некоторого $n$. Описание таких слов
в терминах поворота окружности было получено в работе \cite{C2}.
Для двубуквенных алфавитов они носят название {\it квазиштурмовых}
слов. Слова с функцией роста, удовлетворяющей соотношению $\lim_{n
\to \infty} T(n)/n = 1$ изучены в работе \label{Aberkane}
\cite{Ab}. Слова с функцией роста, удовлетворяющей соотношению
$\lim_{n\to\infty} T(n)/n = 1$ изучены в работе \cite{Ab}.

Можно рассмотреть графы Рози с большим числом развилок. Им
отвечают слова, у которых функция сложности имеет ассимптотику
вида $an+b$. Слова с функцией сложности $T_W(n)=2n+1$ изучены в
работах P.Arnoux, G.Rauzy (\cite{AR,Ra}), с функцией роста
$T_W(n)=2n+1$ в работе G.Rote \cite{R}. Изучением свойств слов с
линейным ростом числа подслов также производилось школой
V.Berth\'e, S.Ferenczi и Luca Q.Zamboni (\cite{FZ}, \cite{BFZ}).
Продвижение в задачах символической динамики для слов с линейной
функцией роста получено в работе \cite{AR} \label{ArnouxRazy}  В
этой работе построена динамическая система для слов с функцией
роста $T(n)=2n+1$, обладающих дополнительным комбинаторным
свойством. В работе \label{Rote} \cite{R} в терминах эволюции
графов Рози описаны слова с функцией роста $2n$.

Рассмотрение общего случая слов у которых графы Рози имеют большее
число развилок (но все же их число конечно), т.е. слов с линейной
функцией сложности приводит к изучению слов, порождаемых {\bf
перекладыванием отрезков.} Известно, что если перекладывание $k$
отрезков {\it регулярно}, то есть траектория любого из концов
отрезка перекладывания не попадает на конец другого отрезка, то
слово, порождаемое данным перекладыванием, имеет функцию сложности
$T(n)=n(k-1)+1$.

Перекладывания отрезков т.е. кусочно непрерывные преобразования
одномерного комплекса естественным образом служат обобщением
вращения круга. (Сдвиг окружности, по сути, является
перекладыванием двух отрезков с сохранением ориентации.) Эти
преобразования были введены Оселедецом
 \cite{Os}, следовавшим идее В.~И.~Арнольда \cite{Arn}, (см. также \cite{KS}).
Рози \cite{Ra3} впервые показал, что связь между вращениями круга
и последовательностями Штурма обобщается если рассматривать
перекладывания отрезков. В связи с этим (в той же работе) он задал
вопрос описания последовательностей, связанных с перекладываниями
отрезков.

Такие последовательности являются еще одним естественным
обобщением слов Штурма. В частном случае, для $k=3$ отрезков,
описание таких последовательностей было получено в работе
\cite{FHZ}, а работе \cite{FZ2} были изучены частные случаи
последовательностей, порождаемых перекладыванием $4$-х отрезков.
Стоит также отметить работы \cite{AHP,Ba1,Ba2,BMP}.

В случае произвольного числа отрезков также получен ряд интересных
результатов. В работе \cite{FZ} получен комбинаторный критерий на
порождаемость слов, получаемых симметричным перекладыванием
отрезков, то есть перекладыванием, связанным с перестановкой
$(1\to k,2\to k-1,\ldots,k\to 1)$.

Опишем комбинаторный критерий (для общего случая, не обязательно
являющегося регулярным) того, что данное сверхслово является
перекладыванием отрезков \cite{BelovChernInt}. В работе \cite{FZ3}
был независимо  получен другой критерий порождаемости слов
преобразованием перекладывания отрезков, удовлетворяющих
следующему условию: траектория каждой концевой точки отрезка
перекладывания не попадает на концевую отрезка перекладывания, в
том числе сама на себя. В этом случае, как не сложно видеть, слова
будут иметь функцию сложности $T(n)=(k-1)n+1$.

Рассмотрим соответствие между подсловами и подмножествами $M$.
Легко видеть, что если начальная точка принадлежит множеству
$U_i$, то ее эволюция начинается с символа $a_i$. Рассмотрим
образы множеств $U_i$ при отображениях $f^{(-1)}, f^{(-2)},
\ldots$. Ясно, что если точка принадлежит множеству
$$T^{(-n)}(U_{i_n}) \cap T^{-(n-1)}(U_{i_{n-1}})
\cap \ldots \cap T^{(-1)}(U_{i_1}) \cap U_{i_0},$$ то эволюция
начинается со слова $a_{i_0}a_{i_1}\cdots a_{i_n}$.
Соответственно, количество различных существенных эволюций длины
$n+1$ равно количеству разбиений множества $M$ на непустые
подмножества границами подмножеств $\partial U_i$ и их образами
при отображениях $f^{-1}, f^{-2}, \ldots , f^{-n+1}$. Обозначим
через $I_u$ множество разбиения, которое соответствует слову $u$.
Ясно, что специальным подсловам соответствуют те интервалы,
которые делятся образами концов перекладываемых интервалов. Для
данного слова $u$ назовем слово $v$ {\it левым (соответственно,
правым) потомком}, если $u$ -- суффикс (соответственно, префикс)
слова $v$, в соответствии с этим будем называть вершину в $G_n$
левым (соответственно, правым) потомком вершины в $G_k$, $n>k$.
Прообраз конца интервала может являться граничной точкой только
для двух интервалов, соответственно, специальные подслова могут
иметь валентность только равную $2$.

\medskip
{\bf Правило 1.\ }\ {\it  Для того, чтобы бесконечное слово $W$
порождалось системой $(I,T,U_1,\ldots,U_k)$ необходимо, чтобы
любое специальное слово имело валентность $2$. }
\medskip

Таким образом, мы можем наложить условие на эволюцию графов Рози:
начиная с некоторого $k$ все $k$-графы Рози имеют входящие и
исходящие развилки степени $2$. Предположим, что некоторому
подслову $w$ соответствует характеристический интервал, полностью
лежащий внутри интервала перекладывания. Пусть точка $A\in [0,1]$
делит $I_w$ на два интервала, образы которых лежат в  $I_{a_k}$  и
$I_{a_l}$ соответственно, а точка $B\in [0,1]$ -- делит на
интервалы, прообразы которых лежат в $I_{a_i}$ и $I_{a_j}$
соответственно.

Выбор минимального невстречающегося  слова, а, значит, удаляемого
ребра, определяется взаиморасположением точек $A$ и $B$, а также
сохранением или сменой ориентации отображения на этих множествах.
Итого, имеется 8 вариантов, которые разбиваются на четыре пары,
соответствующие одинаковым наборам слов. Например, слову $a_i w
a_k$ соответствует ситуация

$$
B<A, T^{-1}([x_w,B])\subset I_{a_i} , T([x_w,A]\subset I_{a_k} ).
$$

 Граф Рози называется
{\it размеченным}, если

\begin{enumerate}
\item Ребра каждой развилки помечены символами $l$ (``left'') и
$r$ (``right'')

\item Некоторые вершины помечены символом ``--''.
\end{enumerate}

{\it Последователем} размеченного графа Рози назовем
ориентированный граф, являющийся его последователем как графа
Рози, разметка ребер которого определяется по правилу:
\begin{enumerate}

\item Ребра, входящие в развилку должны быть помечены теми же
символами, как и ребра, входящие в любого левого потомка этой
вершины;

\item Ребра, выходящие из развилки должны быть помечены теми же
символами, как и ребра, выходящие из любого правого потомка этой
вершины;

\item Если вершина помечена знаком ``--'', то все ее правые
потомки также должны быть помечены знаком ``--''.

\end{enumerate}

\medskip
{\bf Замечание.} Поясним смысл разметки графа. Пусть ребра
входящей развилки соответствуют $a_i$ и $a_j$, символы $l$ и $r$
соответствуют левому и правому множеству в паре
$(T(I_{a_i}),T(I_{a_j}))$. Если символы $a_k$  и $a_l$
соответствуют ребрам исходящей развилки, то символы $l$ и $r$
ставятся в соответствии с порядком ``лево-право'' в паре
$(I_{a_k},I_{a_l})$. Знак ``--'' ставится в вершине, если
характеристическое множество, ей соответствующее, принадлежит
интервалу перекладывания, на котором меняется ориентация.
\medskip

\medskip
{\bf Правило 2 (условие для перехода от графа $G_n$ к $G_{n+1}$).}
\begin{enumerate}
\item Если в графе нет двойных развилок, соответствующих
биспециальным подсловам, то при переходе от $G_n$ к $G_{n+1}$
имеем $G_{n+1}=D(G_n)$;

\item Если вершина, соответствующая биспециальному слову не
помечена знаком ``--'', то ребра, соответствующие запрещенным
словам выбираются из пар $lr$ и $rl$ \item Если вершина помечена
знаком ``--'', то удаляемые ребра должны выбираться из пары $ll$
или $rr$.
\end{enumerate}

Назовем эволюцию размеченных графов Рози {\it правильной}, если
{\bf правила 1} и {\bf 2} выполняются для всей цепочки эволюции
графов, начиная с $G_1$, назовем эволюцию {\it асимптотически
правильной}, если {\bf правила 1} и {\bf 2} выполняются, начиная с
некоторого $G_n$. Будем говорить, что эволюция размеченных графов
Рози {\it ориентирована}, если в $k$-графах нет вершин, помеченных
знаком ``--''.

\begin{theorem}    \label{ThBelCher}
Равномерно-рекуррентное слово $W$
 \begin{enumerate}
\item  Порождается перекладыванием отрезков, тогда и только тогда,
когда слово обеспечивается асимптотически правильной эволюцией
размеченных графов Рози.

\item  Порождается перекладыванием отрезков  с сохранением
ориентации тогда и только тогда, когда слово обеспечивается
асимптотически правильной ориентированной эволюцией размеченных
графов Рози.
\end{enumerate}
\end{theorem}

Графы Рози представляются комбинаторным языком, наиболее
приспособленным для задания одномерных динамических систем. Ещё
одним комбинаторным языком является язык подстановок.
Подстановочные динамические системы и тесно связанные с ними
$D0L-$системы интенсивно исследовались рядом авторов. В этой связи
следует обратиться к характеризации слов Штурма через свойство
(4). Что можно сказать про динамические системы для произвольной
системы подстановок? В частности, нас интересует случай
периодичности системы подстановок. Для поворота окружности
периодичность может наблюдаться только для случая ранга 2, т.е.
для подстановок собственные значения соответствующей матрицы
которых есть квадратичная иррациональность. Для перекладывания
отрезков в общем случае возможно появление высших
иррациональностей. В этой связи важно уметь переводить свойство
периодичности на язык графов Рози.

Эту возможность предоставляет теорема, доказываемая в настоящей статье:

\begin{theorem}       \label{ThMain}
Язык непериодичного равномерно-рекуррентное слово является
подстановочным тогда и только тогда, когда протокол
детерминированной эволюции его схем Рози периодичен, возможно, с
предпериодом.
\end{theorem}

Обратимся к условиям (4) и (5) для графов Рози слов Штурма. Для
случая одной входящей (соответственно, одной исходящей) развилки
периодичность событий в схеме Рози означает периодичность
разложения в цепную дробь числа $\alpha$.
 Поэтому теорему  \ref{ThMain} можно
рассматривать как обобщение теоремы Лагранжа о периодичности
разложения квадратичной иррациональности в цепную дробь на высшие
иррациональности. В случае размеченных схем Рози для
перекладывания отрезков периодичность также равносильна
подстановочности.

Из теоремы \ref{ThMain} также вытекает обобщение классической
теоремы А.А.Маркова. Имеется счетная последовательность
алгебраических констант $\{\alpha_i\}\to 1/3$ и чисел
$\{\beta_i\}$ такие, что для любого алгебраического числа
$\gamma$, не эквивалентного $\beta_j$ для всех $j<i$ неравенство
$|\gamma-p/q|<1/cq^2$ имеет конечное число решений для всех
$c<\alpha_i$. Величина наилучшего приближения определяется
отношением минимального интервала к максимальному в процессе
алгоритма Евклида на паре отрезков. Индукция Рози обобщает
Алгоритм Евклида. Для перекладывания $\le 4$ отрезков
соответствующие результаты были получены в работе \cite{Ferenci2}.
Отношение длин отрезков равно отношению частот соответствующих
подслов. Назовем такое отношение {\em дискретной частью спектра},
если такое же или меньшие отношения достигаются на счетном
множестве эволюций схем Рози. Аналогично определяется {\em
дискретная часть спектра} для перекладывания отрезков.

\begin{corollary}[Теорема Маркова для произвольного числа развилок]
Дискретная часть спектра достигается на периодических схемах Рози.
Соответствующие отношения суть алгебраические числа. Аналогичный
факт верен для размеченных схем Рози реализующий перекладвание
отрезков.
\end{corollary}

Перекладывания отрезков возникают при изучении потоков на
поверхностях отрицательной кривизны, в частности, из изучения
бильярда с рациональными углами. В этом случае также легко
сооружается подстановочная система. Если осуществлять эволюцию
случайным образом, всякий раз выбирая возможность с вероятностью
$1/2$,  то ситуации для слов Штурма отвечает инвариантная
относительно перехода к следующему остатку Гауссова мера, а для
перекладывания отрезков -- инвариантная относительно индукции Рози
мера на пространствах Тейхмюллера.

Из доказательства теоремы \ref{ThMain} вытекает также решение
известного открытого вопроса \cite{Pr,PS}.

\begin{theorem}
Существует алгоритм проверки  равномерно-рекуррентности
морфического слова $h(\varphi^\infty(s))$.
\end{theorem}

Из теоремы \ref{ThMain} вытекает теорема Вершика-Лившица о
периодичности диаграмм Брателли для марковских компактов,
порожденных подстановочными системами \cite{Vershik2,VL}.

Доказательство теоремы Вершика-Лившица основано на явной
конструкции. Рассмотрим образ
$V_n=\varphi^{(n)}(a)=(\varphi^{(n-2)})(\varphi^{(2)}(a))$ где $a$
есть буква алфавита, $\varphi$ есть некоторый примитивный морфизм
полугруппы слов. $V_n$ состоит из блоков отвечающих применению
$\varphi^{(n-2)}$ к буквам слова $\varphi^{(2)}(a)$. Кроме того,
$V_n$ можно представить в виде произведения блоков отвечающих
применению $\varphi^{(n-1)}$ к буквам слова $\varphi(a)$. Конечные
множества, образующие диаграммы Брателли состоят из
последовательностей пар первого рода (блок, его позиция в блоке на
единицу большего размера), которые отвечают собственным подсловам
блоков на единицу б\'ольшего размера (все блоки одного уровня) а
также последовательности второго рода -- последовательность
первого рода для размера $n$ начинающаяся или заканчивающаяся
последовательностью первого рода для предыдущего размера. При этом
естественная замена этой дополнительной последовательности на пару
(соответствующий блок чьим концом или началом она является, его
позиция) естественным образом приводит к появлению
последовательности первого рода. (Особо надо рассмотреть случай
когда при заполнении возникает блок на единицу большего рода чем
все присутствующие). Естественным образом на этих множествах
вводятся стрелки, означающие что соответствующий объект есть
начальное или концевое подслово другого объекта. Детали
конструкции -- см.  \cite{Vershik2,VL} а также
\cite{LifshicDisser}.

Строгое определение схем Рози вводится в разделе 2, определение
эволюции схем Рози и протокола эволюции -- в разделе 5. Для
некоторых подстановочных систем схемы Рози отвечают графам Рози, в
которых простые пути между развилками заменяются на
ориентированные рёбра длины $1$. В частности, такова ситуация в
случае {\it равноблочных маркированных циркулярных}
DOL-последовательностей, изученном в \cite{Frid}.

В общем случае {\it схема Рози} являет собой граф, различные части
которого взяты из графов Рози разных порядков, на рёбрах графа по
некоторому правилу написаны слова.

Доказательство того факта, что если  сверхслово $W$ имеет
периодичную последовательность схем Рози, то оно порождается
подстановкой, не столь сложное. Главная трудность заключается в
обратной импликации. Подстановка может быть сложно устроена. В
частности, образы различных букв могут содержать друг друга.

Значительно более сложной, чем в теореме Вершика-Лившица, частью
рассуждений является доказательство периодичности схем Рози для
морфических слов. Подстановка может быть плохо устроена: образы
различных букв могут содержать друг друга. Для некоторых
подстановочных систем схемы Рози отвечают графам Рози, в которых
простые пути между развилками заменяются на ориентированные рёбра
длины $1$. В частности, такова ситуация в случае {\em равноблочных
маркированных циркулярных} $DOL$-последовательностей, изученном в
\cite{Frid}. В общем случае {\em схема Рози} являет собой граф,
различные части которого взяты из графов Рози разных порядков, на
рёбрах графа по некоторому правилу написаны слова. Назовем  {\em
весом ребра} в схеме Рози длину соответствующего слова. Ключевым
местом является доказательство того, что {\it если слово $W$
порождено примитивным морфизмом, то отношения весов во всех схемах
Рози ограничены}. (То же верно для р.р. морфических слов.) Далее
можно воспользоваться результатом J.Cassaigne \cite{Cassaigne} о
том что если $W$ равномерно рекуррентно и $\liminf
T_W(n)/n<\infty$, то $\limsup T_W(n+1)-T_W(n)<\infty$ и установить
ограниченность числа вершин в схемах Рози, получающихся путем
элементарной последовательности эволюций из заданной. (Переход к
равномерно рекуррентным словам, порождённых {\it произвольным
морфизмом}, вытекает из работ Ю.~Л.~Притыкина.)

Доказательство ограниченности отношений весов основан на
линейности функции возвращения. В процессе доказательства очень
важно отслеживать, что несравнимым по включению путям отвечают
несравнимые по включению слова.

Следует отметить необходимость изучения динамических систем
размерности два и выше. Обратные задачи символической динамики,
связанные с унипотентным преобразованием тора, изучались в работе
\cite{BK}. В частности, такие слова получаются взятием дробных
частей многочленов со старшим иррациональным коэффициентом в целых
точках, однако обсуждение такого рода моделей выходит за рамки
настоящей статьи.

\section{Основные определения.}

{\it Слово} $u$ над конечным алфавитом $\{a_i\}$ -- это последовательность букв: $a_{i_1}a_{1_2}\dots a_{i_n}$.
Слова бывают {\it конечными} и {\it бесконечными}, бесконечные слова бывают {\it бесконечными вправо}, {\it бесконечными влево} или {\it бесконечными в обе стороны}. Бесконечные слово также будет называться {\it сверхсловом}. На буквах слова можно ввести нумерацию, будем требовать, чтобы в словах, не бесконечных слева, нумерация совпадала с естественной нумерацией от левого конца. Тот факт, что на $i$-том месте в слове $u$ стоит буква $a$, будем обозначать $u[i]=a$.

Для конечного слова определена {\it длина} -- количество букв в нём. Длина слова $u$ также будет обозначаться $|u|$.
Если слово $u_1$ не бесконечно справа, а $u_1$ -- не бесконечно слева, то определена их {\it конкатенация} $u_1u_2$ -- слово, получающееся приписыванием второго к первому справа.

Слово $v$ является {\it подсловом} слова $u$, если $u=v_1vv_2$ для некоторых слов $v_1$, $v_2$. В случае, когда $v_1$ или $v_2$ -- пустое слово, $v$ называется {\it началом} или соответственно {\it концом} слова $u$. На словах существует естественная структура частично упорядоченного множества: $u_1\sqsubseteq u_2$, если $u_1$ является подсловом $u_2$. Будем обозначать $u_1\sqsubseteq_k u_2$, если слово $u_1$ входит в $u_2$ хотя бы $k$ раз.

Бесконечное вправо сверхслово $W$ называется {\it рекуррентным}, если любое его подслово встречается в $W$ бесконечно много раз, иначе говоря, $v\sqsubseteq W\Rightarrow v\sqsubseteq _{\infty}W$.

Бесконечное вправо слово $W$ называют {\it равномерно рекуррентным}, если для любого его подслова $v$ существует такое число $k(v,W)$, что если $u\sqsubseteq W$ и $|u|\geq k(v,W)$, то $v\sqsubseteq u$.

Слово $W$ называется {\it периодичным} с периодом k, если для любого $i$ выполнено $W[i+k]=W[i]$.
Слово называется {\it заключительно периодичным} с периодом $k$, если для любого $i$ начиная с некоторого $N$ выполнено $W[i+k]=W[i]$.

Множество слов $A^{*}$ над алфавитом $A$ можно считать свободным моноидом с операцией конкатенации и единицей -- пустым словом. Отображение $\varphi\colon A^{*}\to B^{*}$ называется {\it морфизмом}, если оно сохраняет операцию моноида. Очевидно, морфизм достаточно задать на буквах алфавита $A$. Морфизм называется {нестирающим}, если образом никакой буквы не является пустое слово.

Если адфавиты $A$ и $B$ совпадают и существует такая буква $a_1$, что $\varphi(a_1)=a_1u$ для некоторого слова $u$ и $\varphi^k(u)$ не является пустым словом ни для какого $k$, то бесконечное слово 
$$
a_1u\varphi(u)\varphi^2(u)\varphi^3(u)\varphi^4(u)\dots
$$
называется {\it чисто морфическим}, пишут $W=\varphi^{\infty}(a_1)$.

Если существует такая степень морфизма $\varphi^k$, что для любых двух букв $a_i$ содержится в $\varphi^k(a_j)$, морфизм называют {\it примитивным}.

Для бесконечного вправо слова $W$ определёны {\it графы Рози}. Граф Рози {\it порядка $k$} обозначается $G_k(W)$, если же понятно, о каком слове идёт речь, то будем писать просто $G_k$.
Вершины его соответствуют всевозможным различным подсловам длины $k$ сверхслова $W$.
Две вершины графа $u_1$ и $u_2$ соединяются направленным ребром, если в $W$ есть такое подслово $v$, что $|v|=k+1$, $v[1]v[2]\dots v[k]=u_1$ и $v[2]v[3]\dots v[k+1]=u_2$.
Если $w$ -- подслово $W$ длины $k+l$, ему соответствует в $G_k$ путь длины $l$, проходящий по рёбрам, соответствующим подсловам слова $w$ длины $k+1$.

{\it Графом со словами} будем называть связный ориентированный граф, у которого на каждом ребре которого написано по два слова -- {\it переднее} и {\it заднее}, а кроме того, каждая вершина либо имеет входящую степень $1$, а исходящую больше $1$, либо входящую степень больше $1$ и исходящую степень $1$. Вершины первого типа назовём {\it раздающими}, а второго -- {\it собирающими}.

{\it Путь} в графе со словами -- это последовательность рёбер, каждое следующее из которых выходит из той вершины, в которую входит предыдущая. {\it Симметричный путь} -- это путь, первое ребро которого начинается в собирающей вершине, а последнее ребро кончается в раздающей.

Каждый путь можно записать словом над алфавитом -- множеством рёбер графа, которое называется {\it рёберной записью пути}. Иногда мы будем отождествлять путь и его рёберную запись. Ребро пути $s$, идущее $i$-тым по счёту, будем обозначать $s[i]$. Для двух путей так же, как и для слов, определены отношения {\it подпути} (пишем $s_1\sqsubseteq s_2$), {\it начала} и {\it конца}.
Кроме того, пишем $s_1\sqsubseteq_k s_2$, если для соответствующих слов $u_1$ и $u_2$ -- рёберных записей путей $s_1$ и $s_2$ --  выполнено $u_1\sqsubseteq_k u_2$.
Если последнее ребро пути $s_1$ идёт в ту же вершину, из которой выходит первое ребро пути $s_2$, путь, рёберная запись которого является конкатенацией рёберных записей путей $s_1$ и $s_2$, будем обозначать $s_1s_2$.

\begin{remark}
Определения пути, подпути, рёберной записи, начала и конца имеют смысл для любых графов ориентированных графов. 
\end{remark}

Введём понятие {\it переднего слова $F(s)$, соответствующего пути $s$ в графе со словами.} Пусть $v_1v_2\dots v_n$ -- рёберная запись пути $s$. В $v_1v_2\dots v_n$ возьмём подпоследовательность: включим в неё $v_1$, а также те и только те рёбра, которые выходят из раздающих вершин графа. Эти рёбра назовём {\it передними образующими для пути $s$}. Возьмём передние слова этих рёбер и запишем их последовательную конкатенацию, там получаем $F(s)$.

Аналогично определяется $B(s)$. В $v_1v_2\dots v_n$ возьмём рёбра, входящие в собирающие вершины и ребро $v_n$ в порядке следования -- это {\it задние образующие для пути $s$}. Тогда последовательной конкатенацией задних слов этих рёбер получается
{\it заднее слово $B(s)$ пути $s$}.

\begin{definition} \label{Def1}
Граф со словами будет являться {\it схемой Рози} для рекуррентного непериодичного сверхслова $W$, если он удовлетворяет следующим свойствам, которые в дальшейшем будут называться {\it свойствами схем Рози}:

\begin{enumerate}
    \item Граф сильносвязен и состоит более чем из одного ребра.
    \item Все рёбра, исходящие из одной раздающей вершины графа, имеют передние слова с попарно разными первыми буквами.
	Все рёбра, входящие в одну собирающую вершину графа, имеют задние слова	с попарно разными последними буквами.
    \item Для любого симметричного пути, его переднее и заднее слова совпадают. То есть можно говорить просто о слове 		симметричного пути.
    \item Если есть два симметричных пути $s_1$ и $s_2$ и выполнено $F(s_1)\sqsubseteq _k F(s_2)$, то $s_1\sqsubseteq _k s_2$.
	\item Все слова, написанные на рёбрах графа, являются подсловами $W$.
	\item Для любого $u$~-- подслова $W$ существует симметричный путь, слово которого содержит $u$.  
	\item Для любого ребра $s$ существует такое слово $u_s$, принадлежащее $W$, что любой симметричный путь, слово которого содержит $u_s$, проходит по ребру $s$.
\end{enumerate}  
\end{definition}

\section{Свойства схем Рози.}

Пусть $S$ -- схема Рози для сверхслова $W$.

\begin{lemma}
\begin{enumerate}
\item Если путь $s_1$ оканчивается в раздающей вершине и в этой же вершине начинается путь $s_2$, то $F(s_1s_2)=F(s_1)F(s_2)$.

\item Если путь $s_1$ оканчивается в собирающей вершине и в этой же вершине начинается путь $s_2$, то $B(s_1s_2)=B(s_1)B(s_2)$.
\end{enumerate}
\end{lemma}

\begin{proof}
1.Множество образующих передних рёбер пути $s_1s_2$ -- это в точности образующие передние рёбра пути $s_1$ и образующие передние рёбра пути $s_2$, записанные последовательно.

2.Множество образующих задних рёбер пути $s_1s_2$ -- это образующие задние рёбра пути $s_1$ и образующие задние рёбра пути $s_2$, записанные последовательно.
\end{proof}

\begin{remark}
При доказательстве этой леммы никак не использовалось, что $S$ -- схема Рози. То есть утверждение верно для любого графа со словами.
\end{remark}

\begin{corollary} \label{C2_3}
Если в графе со словами $S$ выполнено свойство $3$ схем Рози, $s$ -- симметричный путь в $S$, $s_1$ -- произвольный путь в $S$ и $s_1\sqsubseteq s$, то $F(s_1)\sqsubseteq F(s)$. 
\end{corollary}

\begin{lemma}
Если $s_1$ и $s_2$ -- два симметричных пути таких, что $s_1\sqsubseteq _k s_2$, то $F(s_1)\sqsubseteq _k F(s_2)$.
\end{lemma}

\begin{proof}
Пусть $v_1v_2\dots v_n$ -- рёберная запись пути $s_2$, а $w_1w_2\dots w_m$ -- рёберная запись $s_1$.
Для каждого вхождения слова $w_1w_2\dots w_m$ в $v_1v_2\dots v_n$ ту часть $v_1v_2\dots v_n$, которая идёт до последнего ребра вхождения (включительно), назовём {\it путевым началом}, а всё, что после -- {\it путевым концом}. Так как ребро $w_m$ входит в раздающую вершину, то любое путевое начало является симметричным путём, а первое ребро любого путевого конца является в пути $s_2$ передним образующим ребром.
Если $s_b$ и $s_e$ -- путевые начало и конец соответственно, то $F(s_2)=F(s_b)F(s_e)$, при этом $F(s_b)$ оканчивается на $F(s_1)$. Для доказательства леммы достаточно показать, что передние слова всех путевых концов различные.
Для любого путевого конца $s_e$ множество его передних образующих рёбер -- это подмножество передних образующих рёбер пути $s_2$, пересечённое с $s_e$. Но для различных путевых концов такие подмножества вложены одно в другое, а так как для любого путевого конца его первое ребро является образующим, то они вложены строго.
\end{proof}

\begin{remark}
Доказательство леммы проходит для любого графа со словами, для которого выполнено свойство $3$ схем Рози.
\end{remark}

\begin{definition}
Симмметричный путь $s$ называется {\it допустимым}, если $F(s)\sqsubseteq W$.
\end{definition}

\begin{lemma} \label{Lm2_4}
Если $u\sqsubseteq W$, то в схеме $S$ есть допустимый путь $s$ такой, что $u\sqsubseteq F(s)$.
\end{lemma}

\begin{proof}
Пусть $l_{max}$ -- максимальная из длин слов на рёбрах $S$.
Так как сверхслово $W$ рекуррентно, то в нём есть вхождение слова $u$ такое, что первая буква этого вхождения имеет в $W$ номер
 более $l_{\max}$. Рассмотрим $w$ -- подслово $W$, имеющее вид $u_1uu_2$, где $|u_1|>l_{max}$, $|u_2|>l_{max}$.
Согласно свойству $6$, в схеме $S$ существует симметричный путь $s_1$ такой, что $w\sqsubseteq F(s_1)$.
Можно считать, что $s_1$ -- минимальный (относительно $\sqsubseteq $) симметричный путь с таким свойством. 

Пусть $v_1v_2\dots v_n$ -- рёберная запись этого пути. Пусть $v_{n_1}$ -- последнее ребро, являющееся передним образующим для пути $s_1$.
Это либо $v_1$, либо ребро, выходящее из раздающей вершины. Если верно первое, то в пути $s_1$ только одно переднее образующее ребро и $|F(s_1)|\leq l_{\max}$.
Значит, можно рассмотреть симметричный путь с рёберной записью $v_1v_2\dots v_{n_1-1}$.
Из минимальности $s_1$ следует, что $w \npreceq F(v_1v_2\dots v_{n_1-1})$. 

Рассуждая аналогично, получим, что если $v_{n_2}$ -- первое ребро, являющееся задним образующим, то $w \npreceq F(v_{n_2}v_{n_2+1}\dots v_n)$.
Докажем следующий факт: $n_2\leq n_1$. В самом деле, иначе, в силу минимальности $n_2$, среди рёбер $v_1,v_2,\dots v_{n_1}$ ни одно не входит в собирающую вершину, то есть в симметричном пути $v_1v_2\dots v_{n_1-1}$ ровно одно заднее образующее ребро (а именно $v_{n_1-1}$) и, стало быть, $|B(v_1v_2\dots v_{n_1-1})|\leq l_{\max}$. Но 
\[
F(s_1)=F(v_1v_2\dots v_{n_1-1})F(v_{n_1})=B(v_1v_2\dots v_{n_1-1})F(v_{n_1})\leq 2 l_{max}.
\]
Противоречие с тем, что $|F(s_1)|>2 l_{\max}$.

Значит, мы можем рассматривать симметричный путь $s_2=v_{n_2}v_{n_2+1}\dots v_{n_1-1}$. Пусть $w' = F(s_2)$.
Мы можем записать $F(s_1)=B(v_{n_2-1})w'F(v_{n_1})$. Так как слово $F(s_1)$ содержит $w$, а слова $B(v_{n_2-1})w'$ и $w'F(v_{n_1})$ не содержат, то $w'\sqsubseteq w$. Стало быть, $w'\sqsubseteq W$.

С другой стороны, $F(s_1)=u'_1u_1uu_2u'_2$, где $\min\{|u'_1u_1|,|u_2u'_2|\}\geq l_{\max}$.
А так как $\max\{B(v_{n_2-1}),F(v_{n_1})\}\leq l_{\max}$, то $u\sqsubseteq w'$. Следовательно, путь $s_2$ -- искомый.

\end{proof}

\begin{remark} \label{N2_5}
В доказательстве леммы \ref{Lm2_4} не использовалось свойство $7$ схем Рози.
\end{remark}


\begin{lemma} \label{Lmtochn}
Пусть имеется допустимый путь $l$ со словом $u$. Если $uu_1\sqsubseteq W$ для некотогого $u_1$, то в схеме Рози $S$ есть такой допустимый путь $l'$, что его началом является $l$, а его слово начинается с $uu_1$. 
\end{lemma}
\begin{proof}
Согласно лемме {\bf \ref{Lm2_4}} в схеме $S$ существует допустимый путь $l_2$ такой, что $uu_1\sqsubseteq F(l_2)$.
Рассмотрим слово $F(l_2)$. В нём может быть несколько вхождений слова $u$, причём среди них есть такое (допустим, $k-$тое, если считать с конца), после которого сразу идёт $u_1$.

Тогда $l\sqsubseteq_k l_2$. 
Пусть $l_2=s_1ls_2$ (здесь рассматривается $k-$тое с конца вхождение пути $l$).
Рассмотрим $k$-тое с конца вхождение пути $l$ в путь $l_2$.
Первое ребро пути $l$ выходит из собирающей  вершины, а последнее ребро пути $s_2$ входит в раздающую вершину. Следовательно, $ls_2$ -- симметричный путь.
Этот путь является допустимым, так как $B(l_2)=B(s_1)B(ls_2)$ и $B(l_2)\sqsubseteq W$.

Так как $l$ имеет ровно $k$ вхождений в $ls_2$, то, согласно свойству $4$ определения схем Рози \ref{Def1} и лемме {\bf \ref{Lmtochn}}, $u=F(l)$ имеет ровно $k$ вхождений в $F(ls_2)$.
Кроме того, $F(ls_2)$ начинается с $F(l)=u$. Этому свойству удовлетворяет ровно одно окончание слова $F(l_2)$. Значит, $F(ls_2)$ начинается со слова $uu_1$. Стало быть, путь $ls_2$ -- искомый.
\end{proof}

\begin{remark}
Аналогично доказывается следующий факт: пусть имеется допустимый путь $l$ со словом $u$. Если $u_1u\sqsubseteq W$ для некотогого $u_1$, то в схеме Рози $S$ есть такой допустимый путь $l'$, что его концом является $l$, а его слово кончается на $u_1u$.
\end{remark}

\begin{corollary}
Если $u$ -- биспециальное подслово $W$ такое, что оно содержит слово некоторого симметричного пути $l_1$ схемы $S$, то в схеме $S$ существует такой симметричный путь $l$, что $F(l)=u$.
\end{corollary}
\begin{proof} Пусть $u=u_1F(l_1)u_2$. Из биспециальности $u$ следует, что существуют буквы $a_1$ и $a_2$ такие, что $F(l_1)u_2a_1\sqsubseteq W$, $F(l_1)u_2a_2\sqsubseteq W$. Тогда существуют два пути $l_1s_1$ и $l_1s_2$ такие, что слово первого начинается с $F(l_1)u_2a_1$, а второго -- с $F(l_1)u_2a_2$.

Пусть $v_1v_2\dots v_{k-1}v_{k}\dots$ -- рёберная запись $l_1s_1$, а $v_1v_2\dots v_{k-1}v'_{k}\dots$ -- рёберная запись $l_1s_2$, различие происходит в ребре с номером $k$.
Очевидно, рёбра $v_k$ и $v'_k$ выходят из одной и той же вершины. Значит, эта вершина раздающая и путь $v_1v_2\dots v_{k-1}$ -- симметричный.
По свойству $2$ схем Рози (см. определение \ref{Def1}), $F(v_k)$ и $F(v'_k)$ начинаются с разных букв. Так как $F(l_1s_1)$ начинается с $F(v_1v_2\dots v_{k-1})F(v_k)$, а $F(l_1s_2)$ -- с $F(v_1v_2\dots v_{k-1})F(v'_k)$, то $F(v_1v_2\dots v_{k-1})=F(l_1)u_2$. При этом путь $v_1v_2\dots v_{k-1}$ начинается с $l_1$.

Аналогично рассуждая, найдём симметричный путь $w_1w_2\dots w_m$, концом которого является путь $l_1$ и словом которого является $u_1F(l_1)$.
Пусть $w_1w_2\dots w_m=l'l_1$. Тогда путь $l'v_1v_2\dots v_{k-1}$ является искомым.
\end{proof}

\section{Получение схем Рози из графов Рози.}
Пусть $W$ -- рекуррентное непериодичное сверхслово.
Фиксируем натуральное число $k$. Для упрощения дальнейшего считаем, что в
$W$ нет биспециальных слов длины ровно $k$.

Рассмотрим $G_k$ -- граф Рози порядка $k$ для сверхслова $W$. 
\begin{proposition} $G_k$ --
сильносвязный орграф, не являющийся циклом.
\end{proposition}

\begin{proof}
Если $u_1$ и $u_2$ -- подслова $W$ длины $k$, то, в силу рекуррентности сверхслова $W$,
в $W$ найдётся подслово вида $a_{i_1}a_{i_2}\dots a_{i_n}$, где $a_{i_1}a_{i_2}\dots a_{i_k}=u_1$, $a_{i_{n-k+1}}a_{i_{n-k}}\dots a_{i_n}=u_2$.

Тогда слова вида $a_{i_l}a_{i_{l+1}}\dots a_{i_{l+k}}$ имеют длину $k+1$ и соответствуют рёбрам графа $G_k$, образующим путь, соединяющий вершины, соответствующие словам $u_1$ и $u_2$.
Кроме того, если бы $G_k$ был циклом длины $n$, то сверхслово $W$ было бы периодичным с периодом $n$.   
\end{proof}

В силу выбора $k$, в $G_k$ нет вершин с входящей и одновременно исходящей степенью более $1$.
Построим граф со словами $S$, вершинами которого будут специальные вершины графа $G_k$, а рёбра -- простыми цепями, соединяющие специальные вершины графа $G_k$.

\begin{proposition}
Пути в графе $S$ соответствуют путям в $G_k$, начинающимся и кончающимся в специальных вершинах. 
\end{proposition}

Пусть простой путь проходит в графе $G_k$ по вершинам
$$
a_{i_1}a_{i_2}\dots a_{k},\: a_{i_2}a_{i_3}\dots a_{k+1},\: \ldots , a_{i_l}a_{i_{l+1}}\dots a_{i_{l+n-1}},\: \ldots, \:a_{i_{n-k+1}}a_{i_{n-k}}\dots a_{i_n}
$$ 
и соединяет специальные вершины 
$a_{i_1}a_{i_2}\dots a_{k}$ и $a_{i_{n-k+1}}a_{i_{n-k}}\dots a_{i_n}$.
Сопоставим этому пути переднее и заднее слова по следующему правилу:

\begin{enumerate}
\item Если вершина, соответствующая $a_{i_1}a_{i_2}\dots a_{i_k}$, явлется раздающей,
то переднее слово пути -- это $a_{i_{k+1}}a_{i_{k+2}}\dots a_{i_n}$.
\item Если вершина, соответствующая $a_{i_1}a_{i_2}\dots a_{i_k}$, явлется в собирающей,
то переднее слово пути -- это $a_{i_1}a_{i_2}\dots a_{i_n}$.
\item Если вершина, соответствующая $a_{i_{n-k+1}}a_{i_{n-k}}\dots a_{i_n}$, является собирающей,
то заднее слово пути -- это $a_{i_1}a_{i_2}\dots a_{i_n-k}$.
\item Если вершина, соответствующая $a_{i_{n-k+1}}a_{i_{n-k}}\dots a_{i_n}$, является раздающей,
то заднее слово пути -- это $a_{i_1}a_{i_2}\dots a_{i_n}$.
\end{enumerate}

\begin{definition}
Если $S$ -- сильносвязный граф, не являющийся циклом, и $s$ -- путь в графе, то {\it естественное продолжение} пути $s$ {\it вправо} -- это минимальный путь, началом которого является $s$ и который оканчивается в раздающей вершине. 
{\it Естественное продолжение} пути $s$ {\it влево} -- это минимальный путь, концом которого является $s$ и который начинается в собирающей вершине.
\end{definition}

Очевидно, для сильносвязных нецикличных графом естественное продолжение существует всегда и единственно.
Теперь мы готовы написать слова на рёбрах $S$: для каждого ребра в качестве переднего слова берётся переднее слово того пути в $G_k$, который соответствует естественному расширению вправо этого ребра.

Заднее же слово -- это заднее слово пути в $G_k$, соответствующего естественному продолжению влево рассматриваемого ребра. 

\begin{proposition} \label{Prop_1}
В полученном графе сло словами $S$ для любого пути $s$ слово $F(s)$ -- это переднее слово того пути графа $G_k$, который соответствует естественному продолжению вправо пути $s$. Аналогично, $B(s)$ -- это заднее слово того пути в $G_k$, который соответствует естественному продолжннию пути $s$ влево (в графе $S$).
\end{proposition}
\begin{proof}
Действительно, в графе $S$ для любого пути $s$ его естественное продолжение вправо разбивается на естественные продолжения вправо его передних образующих рёбер, причём все естественные продолжения, начиная со второго, выходят из раздающих вершин и, стало быть, слова на них пишутся по правилу для первого из четырёх типов. Значит, слова для этих рёбер в объединении и дадут слово для длинного пути в $G_k$.
\end{proof}

\begin{lemma}      \label{Lm1}
Пусть $S$ -- определённый выше граф со словами. Тогда он является схемой Рози для сверхслова $W$.
\end{lemma}

\begin{proof}

{\bf Свойство 1.}
Так как $G_k$ является сильносвязным и не является циклом, то же самое верно и для $S$.

{\bf Свойство 2.}
Два пути, выходящие из одной раздающей вершины графа $S$ с различными первыми рёбрами, соответствуют путям в графе $G_k$, которые выходят из одной раздающей вершины $a_{i_1}a_{i_2}\dots a_{i_k}$, причём первые рёбра у путей разные.
Пусть вторые вершины путей -- это $a_{i_2}a_{i_3}\dots a_{i_k}b$ и $a_{i_2}a_{i_3}\dots a_{i_k}c$ соответственно.
Тогда слова этих путей начинаются на буквы $b$ и $c$ соответственно, и на эти буквы начинаются слова соответствующих путей в $S$. 

{\bf Cвойство 3} следует из того, что для симметричного пути его естественным продолжением вправо и естественным продолжением влево является он сам.
Согласно \ref{Prop_1}, передним и задним словом такого пути будут соответственно переднее и заднее слово пути
$$
a_{i_1}a_{i_2}\dots a_{k},\: a_{i_2}a_{i_3}\dots a_{k+1},\: \ldots ,
a_{i_l}a_{i_{l+1}}\dots a_{i_{l+n-1}},\: \ldots, \:a_{i_{n-k+1}}a_{i_{n-k}}\dots a_{i_n}
$$ 
графа $G_k$. Переднее слово этого пути определяется по типу $2$, а заднее -- по типу $4$, то есть они совпадают.

Докажем {\bf свойство 4}. Пусть в $S$ нашлись два симметричных пути $s_1$ и $s_2$ такие, что $F(s_1)\sqsubseteq _k F(s_2)$.
Соответственные пути в $G_k$ обозначим $s'_1$ и $s'_2$. Очевидно, достаточно показать, что $s'_1\sqsubseteq _ks'_2$.
пути $s'_1$ и $s'_2$ выходят из собирающих вершин графа $G_k$ и входят в раздающие вершины. Передние слова этих путей -- это $F(s_1)$ и $F(s_2)$. Последовательности их $(k+1)$-буквенных подслов -- это последовательности рёбер $s'_1$ и $s'_2$. Значит, рёберная запись пути $s'_1$ встречается в рёберной записи $s'_2$ хотя бы $k$ раз.

{\bf Свойство 5} достаточно доказать для передних слов. Пусть $v$ -- ребро схемы $S$.
Рассмотрим в $S$ естественное продолжение ребра $v$ вправо. Оно соответствует пути $s$ в $G_k$. Пусть $s$ содержит $L$ рёбер.
Первая вершина этого пути -- слово $a_{i_1}a_{i_2}\dots a_{i_k}$.
Первое ребро пути $s$ соответствует подслову $u$ вида $a_{i_1}a_{i_2}\dots a_{i_k}b$, при этом $a_{i_1}a_{i_2}\dots a_{i_k}b\sqsubseteq W$.
Очевидно, существует $u_1$ -- подслово сверхслова $W$ длины $L+k$, начинающееся с $a_{i_1}a_{i_2}\dots a_{i_k}b$. Ему соответствует путь $s'$ длины $L$ по рёбрам $G_k$ с первым ребром таким же, как у пути $s$.

Так как среди промежуточных вершин пути $s$ нет раздающих вершин (иначе $s$ не соответствует естественному продолжению ребра $v$ вправо), то $s'$ должен совпадать с путём $s$. Стало быль, слово пути $s$ является подсловом $u_1$ и, следовательно, подсловом $W$.

{\bf Свойство 6}. Существует такое число $M$,
что  любом пути длины $M$ по рёбрам $G_k$ будет хотя бы одна раздающая и хотя бы одна собирающая вершина. В самом деле, в качестве $M$ подойдёт количество вершин в $G_k$, иначе существовал бы цикл, в котором не было бы либо раздающих, либо собирающих вершин и $G_k$ был бы либо не сильносвязным, либо циклом.

Пусть $w\sqsubseteq W$. Тогда существует $w'\sqsubseteq W$ вида $w_1ww_2$, где $|w_1|=2M$, $|w_2|=2M$. Рассмотрим в $G_k$ путь, соответствующий слову $w'$.
Среди его последних $M$ вершин есть раздающая, а среди первых $M$ -- собирающая. Значит, можно взять путь, который короче не более, чем на $2M$ и соответствует симметричному пути в $S$. Этот путь соответствует подслову $w'$, которое имеет длину не менее $|w'|-2M$ и, следовательно, содержит $w$.
А это значит, что слово некоторого симметричного пути в $S$ содержит $w$.
 
Докажем {\bf свойство 7}. Пусть $v$ -- ребро в графе $S$. В графе Рози $G_k$ ребру $v$ соответствует путь $v'$, проходящий через вершины
$$
a_{i_1}a_{i_2}\dots a_{i_k},\: a_{i_2}a_{i_3}\dots a_{i_{k+1}},\: \ldots ,
a_{i_l}a_{i_{l+1}}\dots a_{i_{l+n-1}},\: \ldots, \:a_{i_{n-k+1}}a_{i_{n-k}}\dots a_{i_n}.
$$ 
Ни одна из вершин пути кроме первой и последней не является раздающей или собирающей.

Для $S$ уже доказаны свойства $1$--$6$. Согласно замечанию \ref{N2_5}, для $S$ справедлива лемма \ref{Lm2_4}.
Для слова $a_{i_1}a_{i_2}\dots a_{i_k}a_{i_{k+1}}$, являющегося подсловом сверхслова $W$, найдём в $S$ допустимый путь $w$ такой, что $a_{i_1}a_{i_2}\dots a_{i_k}a_{i_{k+1}}\sqsubseteq F(w)$. Соответственный путь в графе $G_k$ содержит первое ребро пути $v'$, а следовательно, и весь $v'$. Стало быть, $w$ содержит $v$, а так как для $S$ выполнено свойство $4$, то любой симметричный путь, содержащий $F(w)$, содержит ребро $v$. 
\end{proof}

\section{Эволюция схем Рози.}
Далее $W$ -- рекуррентное непериодичное сверхслово, $S$ -- схема Рози для этого сверхслова.
Из сильносвязности $S$ следует, что существует ребро $v$, идущее из собирающей вершины в раздающую.
Такие рёбра будем называть {\it опорными}.

Цель этого раздела -- опредедить {\it эволюцию} $(W,S,v)$ схемы Рози $S$ по опорному ребру $v$. Опорное ребро является симметричным путём, стало быть, переднее и заднее слова этого ребра совпадают.

Пусть $\{x_i\}$ -- множество рёбер. входящих в начало $v$, а $\{y_i\}$ -- множество рёбер, идущих из конца $v$. Эти два множества могут пересекаться.
Обозначим $F(y_i)=Y_i$, $B(x_i)=X_i$, $F(v)=V$.

Рассмотрим все слова вида $X_i V Y_j$. Если такое слово не входит
в $W$, то пару $(x_i,y_j)$ назовём {\it плохой}, в противном случае -- {\it хорошей}.

Построим граф $S'$. Он получается из $S$ заменой ребра $v$ на $K_{\#\{x_i\},\#\{y_j\}}$. Более подробно: ребро $v$ удаляется, его начало заменяется на множество из $\{A_i\}$ из $\#\{x_i\}$ вершин так, что для любого $i$ ребро $x_i$ идёт в $A_i$;
конец ребра $v$ заменяется на множество $\{B_j\}$ из $\#\{y_j\}$ вершин так, что для любого $j$ ребро $y_j$ ведёт в $B_j$; вводятся рёбра $\{v_{i,j}\}$, соединяющие вершины множества $\{A_i\}$
с вершинами множества $\{B_j\}$.

По сравнению с $S$, у графа $S'$ нет ребра $v$, но есть новые рёбра $\{v_{ij}\}$. Остальные рёбра графа $S$ имеют соответственные рёбра в $S'$, ребра в первом и во втором графе зачастую будут обозначаться одними и теми же буквами.

На рёбрах $S'$ расставим слова следующим образом. На всех рёбрах $S'$, кроме рёбер из $\{y_i\}$ и $\{v_{i,j}\}$, передние слова пишутся те же, что и передние слова соответственных рёбер в $S$. Для каждого $i$ и $j$, переднее слово ребра $y_j$ в $S'$ -- это $VY_j$. Переднее слово ребра $v_{i,j}$ -- это $Y_j$.
Аналогично, на всех рёбрах, кроме рёбер из $\{x_i\}$ и $\{v_{i,j}\}$, задние слова переносятся с соответствующих рёбер $S$; для всех $i$ и $j$ в качестве заднего слова ребра $x_i$ возьмём $X_iV$, а в качестве заднего слова ребра $v_{i,j}$ -- $X_i$.

Теперь построим граф $S''$. Рёбра $v_{i,j}$, соответствующие плохим парам $(x_i,y_j)$, назовём {\it плохими}, а все остальные рёбра графа $S'$ -- хорошими. Граф $S''$ получается, грубо говоря, удалением плохих рёбер из $S'$.
Более точно, в графе $S''$ раздающие вершины -- подмножество раздающих вершин $S'$, а собирающие вершины -- подмножество собирающих вершин $S'$.
В графе $S'$ эти подмножества -- вершины, из которых выходит более одного хорошего ребра, и вершины, в которые входит более одного хорошего ребра соответственно.
Назовём эти вершины $S'$ неисчезающими. Рёбра в графе $S''$ соответствуют таким путям в $S'$, которые идут лишь по хорошим рёбрам, начинаются в неисчезающих вершинах, заканчиваются в неисчезающих, а все промежуточные вершины которых не являются неисчезающими.

Согласно лемме \ref{Lm4_5}, которая будет доказана ниже, $S''$ -- сильносвязный граф и не цикл. Следовательно, у каждого ребра в $S''$ есть естественное продолжение вперёд и назад. Естественное продолжение вперёд соответствует пути в $S'$; переднее слово этого пути в $S'$ возьмём в качестве переднего слова для соответствующего ребра $S''$. Аналогично для задних слов: у ребра в $S''$ есть естественное продолжение влево, этому пути в $S''$ соответствует путь в $S'$. Заднее слово этого пути и будет задним словом ребра в $S''$.

\begin{definition} Построенный таким образом граф со словами $S''$ назовём {\it элементарной эволюцией} $(W,S,v)$. 
\end{definition}

\begin{theorem} \label{T1}
Пусть $W$ -- схема Рози. Тогда элементарная эволюция $(W,S,v)$ является схемой Рози для сверхслова $W$ (то есть удовлетворяет свойствам $1$---$7$ определения \ref{Def1}).
\end{theorem}
Доказательству этого и будет посвящена основная часть раздела. 

\begin{definition}
Пусть $S$ -- схема Рози для сверхслова $W$. На её рёбрах можно написать различные натуральные числа (или пары чисел). Такую схему мы назовём {\it нумерованной}. Если с рёбер пронумерованной схемы стереть слова, получится {\it облегчённая нумерованная схема}.
\end{definition}

\begin{definition}
{\it Метод эволюции} -- это функция, которая каждой облегчённой пронумерованной схеме (с нумерацией, допускающей двойные индексы) даёт этой же схеме новую нумерацию, такую, что в ней используются числа от $1$ до $n$ для некоторого $n$ по одному разу каждое.
\end{definition}

Зафиксируем какой-либо метод эволюции и далее не будем его менять.

Среди опорных рёбер пронумерованной схемы $S$ возьмём ребро $v$ с наименьшим номером, и совершим эволюцию $(W,S,v)$.
Укажем естественную нумерацию новой схемы.

Напомним, что сначала строится схема $S'$, а потом -- $S''$. В схеме $S'$ все рёбра можно пронумеровать по следующему правилу: рёбра кроме $v$ сохраняют номера, а рёбра вида $v_{ij}$ нумеруют соответствующим двойным индексом.

Каждое ребро в схеме $S''$  -- это некоторый путь по рёбрам схемы $S'$, различным рёбрам из $S''$ соответствуют в $S'$ пути с попарно различными первыми рёбрами. Таким образом, рёбра схемы $S''$ можно пронумеровать номерами первых рёбер соответствующих путей $S'$. Теперь применим к облегчённой нумерованной схеме $S''$ метод эволюции. Получится новая облегчённая нумерованная схема, и в нумерованной (не облегчённой) схеме $S''$ перенумеруем рёбра соответственным образом.

\begin{definition}
Описанное выше соответствие, ставящее нумерованной схеме Рози другую нумерованную схему Рози, назовём {\it детерменированной эволюцией}. Будем обозначать это соответствие $S''=\Evol(S).$
\end{definition}

\begin{definition}
Применяя детерменированную эволюцию к нумерованной схеме $S$ много раз, получаем последовательность нумерованных схем Рози.
Кроме того, на каждом шаге получаем множество пар чисел, задающие плохие пары рёбер. {\it Протокол детермеминованной эволюции} -- это последовательность таких множеств пар чисел и облегчённых нумерованных схем.
\end{definition}

\begin{proposition}
Облегчённая нумерованная схема для $\Evol(S)$ однозначно определяется по облегчённой нумерованной схеме $S$ и множеству пар чисел,
задающие плохие пары рёбер.
\end{proposition}

Сформулируем основной результаты работы:
\begin{theorem}     \label{Th1}
Если сверхслово $W$ является морфическим с примитивным порождающим морфизмам, то протокол детерминированной эволюции с некоторого места периодичен. 
\end{theorem}

В разделе $9$ теорема обобщается на случай произвольного равномерно рекуррентного морфического слова.

Дальнейшая часть раздела посвящена доказательству теоремы \ref{T1}.

Рассмотрим следующее соответствие $f$ между симметричными путями в $S$ за исключением опорного ребра $v$ и симметричными путями в $S'$.
Пусть $s$ -- симметричный путь в $S$ имеет рёберную запись $v_1v_2\dots v_n$. Часть этих рёбер (может быть, ни одно) являются ребром $v$.
Очевидно, если ребро $v$ входит в путь, то оно находится в между рёбрами $x_i$ и $y_j$ для некоторых $i$ и $j$.
Если $v$ является первым или последним ребром $s$, то соответствующее $x_i$ или $y_j$ находится только с одной стороны.
Соответствие $f$ преобразует $v_1v_2\dots v_n$ (рёберную запись $s$) следующим образом:
для каждого вхождения $v$ определяется, находится ли оно на конце рёберной записи или в середине, если на конце, то ``$v$'' просто отбрасывается, а если в середине, то ``$v$'' заменяется на ``$v_{ij}$''.
Индексы $i$ и $j$ берутся те же, что и у рёбер $x_i$ и $y_j$, стоящих рядом с $v$ в рёберной записи $s$. Путь с соответствующей рёберной записью является образом $f(s)$ при соответствии.

Обратное соответствие $f^{-1}$ таково: если  $S'$ есть путь $l$, и его рёберная запись -- $v_1v_2\dots v_n$,
то каждое вхождение ``$v_{ij}$'' заменяется на ``$v$''.
Кроме того, если $v_1$ -- это одно из $\{y_i\}$, то в начало рёберной записи приписывается ``$v$'',
если же $v_n$ -- одно из рёбер $\{x_i\}$, то ``$v$'' приписывается в конец рёберной записи.

Очевидно, $f$ и $f^{-1}$ сохраняют свойство последовательности рёбер ``начало каждого следующего ребра является концом предыдущего'', то есть пути переводят в пути. Кроме того, $f$ и $f^{-1}$ на самом деле являются взаимно обратными. Далее, в графе $S'$ рёбра вида $y_j$ выходят из собирающих вершин, рёбра вида $x_i$ входят в раздающие, рёбра $v_{ij}$ выходят из раздающих вершин и входят в собирающие; в графе $S$ ребро $v$, как опорное, выходит из собирающей и входит в раздающую, рёбра $x_i$ входят в собирающую вершину, а $y_j$ выходят из  раздающей, для всех остальных рёбер графа $S$ свойство выходить из собирающей вершины или входить в раздающую сохраняется в графе $S'$. Таким образом, при соответствиях $f$ и $f^{-1}$ сохраняется симметричность путей.

\begin{proposition} \label{Pr4_2}
При этом соответствии сохраняются слова симметричных путей.
\end{proposition}
\begin{proof}
Докажем этот факт для передних слов, так как для задних всё аналогично. А следовательно, мы докажем, что передние и задние слова симметричных путей в $S'$ совпадают.

Пусть в графе $S$ имеется симметричный путь $s$ с рёберной записью $v_1v_2\dots v_n$.
Разберём два случая.
\begin{enumerate}
\item Пусть первое ребро $v_1=v$, а последнее $v_n \not = v$. Рассмотрим все вхождения ребра $v$:
$$
s = vy_{j_1}\dots vy_{j_2}\dots vy_{j_3}\dots v_n.
$$
Пусть $f(s)=s'$.
$$
s'=y_{j_1}\dots v_{i_2j_2}y_{j_2}\dots v_{i_3j_3}y_{j_3}\dots v_n.
$$
Рассмотрим передние образующие рёбра в $s$ и $s''$. В пути $s$ в парах вида $vy_j$ окажется выбранным ребро $y_j$,
кроме первой пары, где выбраны будут оба ребра. В $s'$ в парах $v_{ij}y_j$ окажутся выбраны рёбра $v_{ij}$,
при этом первое ребро $y_{j_1}$ тоже будет передним образующим. Остальные передние образующие рёбра в обоих путях одни и те же.
Таким образом, 
$$
F(s)=F(v)F(y_{j_1})\dots F(y_{j_2})\dots F(y_{j_3})\dots = VY_{j_1}\dots Y_{j_2}\dots Y_{j_3}\dots
$$
$$
F(s')=F(y_{j_1})\dots F(v_{i_2j_2})\dots F(v_{i_3j_3})\dots = VY_{j_1}\dots Y_{j_2}\dots Y_{j_3}\dots
$$

Фрагменты, отмеченные многоточиями, одинаковы в обоих наборах.
\item Пусть $v_1=v$, $v_n = v$. Аналогично первому случаю, рассмотрим вхождения $v$ в рёберную запись пути $s$.
$$
s = vy_{j_1}\dots vy_{j_2}\dots vy_{j_3}\dots x_{i_k}v.
$$
$$
s'=f(s)=y_{j_1}\dots v_{i_2j_2}y_{j_2}\dots v_{i_3j_3}y_{j_3}\dots x_{i_k}.
$$
Тогда
$$
F(s)=F(v)F(y_{j_1})\dots F(y_{j_2})\dots F(y_{j_3})\dots = VY_{j_1}\dots Y_{j_2}\dots Y_{j_3}\dots
$$
$$
F(s')=F(y_{j_1})\dots F(v_{i_2j_2})\dots F(v_{i_3j_3})\dots = VY_{j_1}\dots Y_{j_2}\dots Y_{j_3}\dots
$$
Как мы видим, слова совпадают.
\item Пусть $v_1 \not = v$, $v_n  = v$. Аналогично, рассмотрим вхождения $v$ в рёберную запись пути $s$.
$$
s = v_1\dots vy_{j_1}\dots vy_{j_2}\dots x_{i_k}v.
$$
$$
s'=f(s)=v_1\dots v_{i_1j_1}y_{j_1}\dots v_{i_2j_2}y_{j_2}\dots x_{i_k}.
$$
Тогда
$$
F(s)=F(v_1)F(y_{j_1})\dots F(y_{j_2})\dots F(y_{j_3})\dots = F(v_1)Y_{j_1}\dots Y_{j_2}\dots Y_{j_3}\dots
$$
$$
F(s')=F(v_1)F(v_{i_1j_1})\dots F(v_{i_2j_2})\dots F(v_{i_3j_3})\dots = F(v_1)Y_{j_1}\dots Y_{j_2}\dots Y_{j_3}\dots
$$
\item Пусть $v_1 \not = v$, $v_n \not = v$. Рассмотрим рёберную запись пути $s$.
$$
s = v_1\dots vy_{j_1}\dots vy_{j_2}\dots v_n.
$$
$$
s'=f(s)=v_1\dots v_{i_1j_1}y_{j_1}\dots v_{i_2j_2}y_{j_2}\dots v_n.
$$
Тогда
$$
F(s)=F(v_1)F(y_{j_1})\dots F(y_{j_2})\dots F(y_{j_3})\dots = F(v_1)Y_{j_1}\dots Y_{j_2}\dots Y_{j_3}\dots
$$
$$
F(s')=F(v_1)F(v_{i_1j_1})\dots F(v_{i_2j_2})\dots F(v_{i_3j_3})\dots = F(v_1)Y_{j_1}\dots Y_{j_2}\dots Y_{j_3}\dots
$$
\end{enumerate}
\end{proof}

\begin{proposition}
Если в графе $S$ два симметричных пути $s_1\not = v$, $s_2\not = v$ и $s_2\sqsubseteq _n s_1$, то $f(s_1)\sqsubseteq _nf(s_2)$.
\end{proposition}
\begin{proof}
Напомним, что $v$ -- это выделенное опорное ребро. Рассмотрим рёберные записи путей $s_1$ и $s_2$.
Вторая рёберная запись входит в первую не менее $n$ раз.

Соответствие $f$ делает с этими записями следующее: каждое вхождение ``$v$'' либо отбрасываются, если оно находилось с краю рёберной записи,
либо меняется на некоторою ребро, зависящее только от соседей ``$v$'' в записи справа и слева.

Рассмотрим некое вхождение второй рёберной записи в первую. Все буквы-рёбра, не являющиеся  ``$v$'', при соответствии $f$ не меняются.
Если буква ``$v$'' находится не на конце рёберной записи $s_2$, то в обеих рёберных записях она меняется на одну и ту же букву.
Если ``$v$'' находится в начале или конце второго слова, при соответствии $f$ во втором слове она исчезает.

Следовательно, при отображении $s$ подпоследовательность рёберной записи пути $s_1$, являющаяся рёберной записью $s_2$, переходит в подпоследовательность рёберной записи $s'$, содержащую рёберную запись $s_2$. Кроме того, так как $s_2\not = v$, в рёберной записи $s_2$ есть первое с начала ребро, не меняющееся при отображении $s$ (то есть не являющееся $v$). Для каждого вхождения второй рёберной записи в первую отметим это ребро.
Рассмотрим образ первого слова при соответствии $f$. Отмеченные рёбра передут в отмеченные рёбра, стоящие в различным местах рёберной записи пути $f(s)$.
При этом образ каждого из отмеченных рёбер будет началом рёберной записи $f(s_2)$.
Таким образом, в рёберной записи первого слова будет не менее $n$ подслов, являющихся рёберной записью пути $f(s_2)$.
\end{proof}

Доказанное утверждение вкупе с предложением \ref{Pr4_2} моментально влечёт тот факт, что для графа $S'$ выполняется свойство $4$ схем Рози.

Далее, из построения $S'$ очевидно, что для любой раздающей вершины передние слова рёбер, выходящих из неё, начинаются с различных букв, а задние слова рёбер, входящих в одну собирающую вершину, кончаются на различнные буквы.

\begin{proposition} \label{Pr4_4}
Для графа со словами $S'$ выполнены свойства 1 -- 4 схем Рози.
\end{proposition}
\begin{proof}
Осталось доказать свойство $1$. В силу сильной связности $S$, в $S$ существует цикл, проходящий по всем рёбрам $S$,
а также по всем тройкам рёбер вида $x_ivy_j$. Если применить к нему отображение $f$, то получится цикл в $S'$, проходящий по всем рёбрам.
То, что $S'$ -- не цикл, очевидно, так как у него есть раздающие рёбра.   
\end{proof}

Составим $\{u_i\}$  -- набор подслов сверхслова $W$.
Для каждого ребра $v\in S$ включим в этот набор слово $u_s$ -- слово, соответствующее этому ребру по свойству $7$ для схем Рози.
Кроме того, включим в набор те слова вида $X_i V Y_j$, которые являются подсловами $W$.
В сверхслове $W$ существует подслово $U_1$, содержащее все слова набора $\{u_i\}$.
Пусть $l_{max}$ -- максимальная из длин слов на ребре $S$.
В $W$ найдётся слово $U_2$ вида $U_1wU_1$, где $|w|>10l_{\max}$.

Применим к $U_2$ лемму \ref{Lm2_4}. Получим допустимый путь $s$ с рёберной записью $v_1v_2\dots v_n$.
Слово $F(s)$ имеет вид $w_1U_1wU_1w_2$.
Множество симметричных путей, являющихся началом $s$, можно упорядочить по включению:
$$
s_1\sqsubseteq s_2\sqsubseteq \dots\sqsubseteq s_k=s.
$$

Если $v_{i_1},v_{i_2},\dots,v_{i_k}$ -- передние образующие для пути $s$, то
$$
F(s_j)=F(v_{i_1})F(v_{i_2})\dots F(v_{i_j}).
$$
Значит, $F(s_{i+1})-F(s_i)\sqsubseteq l_{\max}$ для любого $i$.

Рассмотрим минимальное $i$ такое, что $w_1U_1\sqsubseteq F(s_i)$.
Рёберная запись $s_i$ имеет вид $s_i=v_1v_2\dots v_{k_1}$. 
Очевидно, $F(s_i)=w_1U_1w'_1$ для некоторого $w'_1$. При этом $|w'_1|\leq l_{\max}$,
иначе слово $F(s_{i-1})$, являющееся началом $F(s_i)$ длиной не менее $|F(s_i)|-l_{\max}$, содержало бы $w_1U_1$.

Аналогично, существует такой симметричный путь $s'$, что его рёберная запись имеет вид $v_{k_2}v_{k_2+1},\dots v_n$, а его слово
$F(s')=w'_2U_1w_2$, где $|w'_2|\leq l_{max}$.

Итак, у нас есть пути $s_i=v_1v_2\dots v_{k_1}$ и $s'=v_{k_2}v_{k_2+1},\dots v_n$,
являющиеся соответственно началом и концом пути $s=v_1v_2\dots v_n$.
Докажем, что $k_1<k_2$. Предположим противное: $k_1\geq k_2$.
Введём обозначения: $A =B(v_1v_2\dots v_{k_2-1})$, $B=F(v_{k_2}\dots v_{k_1})$ (этот путь симметричен),
$C=F(v_{k_1+1}v_{k_1+2},\dots v_n)$. Тогда слова путей $v_1v_2\dots v_{k_1}$,
$v_{k_2}v_{k_2+1},\dots v_n$ и $v_1v_2\dots v_n$ -- это $AB$, $BC$ и $ABC$ соответственно. Следовательно, $F(s)<F(s_i)+F(s')$.
С другой стороны, 
$$
F(s)=|w_1U_1wU_1w_2|=|w_1U_1w'_1|+|w'_2U_1w_2|+(|w|-|w'_1|-|w'_2|)\geq F(s_i)+F(s')+8l_{\max}.
$$

Мы получили противоречие. Стало быть, верно неравенство $k_2>k_1$.

Так как $U\sqsubseteq F(s_i)$, то путь $s_i$ проходит по всем рёбрам $S$. Докажем, он проходит и по всем тройкам рёбер вида $x_{i_1}vy_{i_2}$ для хороших пар $(x_{i_1},y_{i_2})$. В самом деле: если $(x_{i_1},y_{i_2})$ -- хорошая пара рёбер, то слово $X_{i_1}VY_{i_2}$ является подсловом $U_1$.
С другой стороны, $X_{i_1}VY_{i_2}$ является словом симметричного пути, образованного конкатенацией естественного продолжения влево ребра $x_{i_1}$, ребра $v$ и естественного продолжения вправо ребра $y_{i_2}$. Осталось воспользоваться свойством $4$ схем Рози, применённым к схеме $S$.

То же самое верно и для пути $s'$.

Если симметричный путь проходит по тройке рёбер $x_{i_1}vy_{i_2}$ для плохой пары $(x_{i_1},y_{i_2})$,
то слово этого пути содержит $X_{i_1}VY_{i_2}$ в качестве подслова и путь не является допустимым. Следовательно, путь $s$, как допустимый, не проходит по ``плохим'' тройкам рёбер. 

\begin{remark} \label{N4_6}
Конструкция пути $f(s)$ в графе $S'$ понадобится далее.
\end{remark}

\begin{lemma} \label{Lm4_5}
Граф $S''$ сильносвязен и не цикл. 
\end{lemma}

\begin{proof}
Рассмотрим в графе $S'$ путь $f(s)$. Этот симметричный путь проходит только по хорошим рёбрам графа $S'$.
В этом пути можно выделить начало $f(s_i)$, конец $f(s')$ и среднюю часть, причём начало и конец проходят по всем хорошим рёбрам графа $S'$.
Пусть $d$ и $e$ -- рёбра схемы $S''$. Им соответствуют пути в $S'$.
Ребру $d$ -- путь $d_1d_2\dots d_{n_1}$, ребру $e$ -- путь $e_1e_2\dots e_{n_2}$.
Путь $f(s_i)$ проходит по ребру $d_1$. Из концов каждого из рёбер $d_1,d_2,\dots,d_{n_1-1}$ идёт ровно одно хорошее ребро. Следовательно, в пути $s$ после ребра $d_1$ обязаны идти рёбра $d_2$,$d_3$,\ldots, $d_{n_1}$. Аналогично, в $s'$ есть вхождение ребра $e_{n_2}$,
а перед вхождением $e_{n_2}$ в $f(s)$ обязана идти последовательность рёбер $e_1e_2\dots e_{n_2-1}$.
Часть пути $f(s)$, начинающуюся с $d_1d_2\dots d_{n_1}$ и кончающуюся на $e_1e_2\dots e_{n_2}$, можно разбить на последовательные группы рёбер, образующие рёбра графа $S''$.

Таким образом, по рёбрам графа $S''$ можно добраться от любого ребра до любого другого. Следовательно, граф $S''$ сильносвязен.

Предположим, что $S''$ -- цикл. Тогда из каждой вершины графа $S'$ выходит ровно одно хорошее ребро.
Пусть всего хороших рёбер $n$. Рассмотрим какое-нибудь хорошее ребро, которое в схеме $S'$ выходит из раздающей вершины, и назовём его $d_1$.
Такие рёбра существуют, например, сойдёт ребро вида $v_{ij}$. В схеме $S'$ хорошие рёбра образуют цикл $d_1d_2\dots d_n$.
Возьмём $w_0$ -- произвольное подслово $W$. В $S$ cуществует симметричный путь $v_1v_2\dots v_n$ такой, что $w_0\sqsubseteq F(v_1v_2\dots v_n)\sqsubseteq W$.
Соответствующий путь в $S'$ должен проходит только по хорошим рёбрам -- следовательно, этот путь имеет вид
$d_{begin}(d_1d_2\dots d_e)^kd_{end}$, где участки $d_{begin}$ и $d_{end}$ состоят менее чем из $n$ рёбер.
Пусть переднее слово пути $d_1d_2\dots d_n$ -- это $D$. Тогда $w_0$ является подсловом слова $D_{begin}D^kD_{end}$, где длина слов $D_{begin}$ и $D_{end}$ не превосходит $n l_{max}$.
Пусть $N$ -- длина слова $D$. Тогда любое подслово сверхслова $W$ длины $N$ является циклическим сдвигом $D$ (если $D'$ -- подслово $W$ длины $N$, выберем $w_0$, содержащее $D$ в середине и имеющее длину $N+10\cdot n l_{max}$).
Таким образом, $W$ периодично. Противоречие, значит, $S''$ -- не цикл.
\end{proof}

\begin{proposition}
Для $S''$ выполнено свойство $2$ схем Рози.
\end{proposition}
\begin{proof}
В самом деле, рёбра, выходящие из одной раздающей вершины в графе $S''$, соответствуют путям,
выходящим из одной раздающей вершины графа $S'$ и имеющим различные первые рёбра. Стало быть, передние слова рёбер в $S''$ -- передние слова некоторых путей, выходящих из одной раздающей вершины графа $S'$ и имеющих различные первые рёбра. Передние слова этих путей начинаются с передних слов соответствующих первых рёбер, первые буквы которых, согласно предложению \ref{Pr4_4}, различны.
\end{proof}

Так как симметричные пути в $S''$ являются симметричными путями в $S'$ с теми же словами, а для графа $S'$ выполнены свойства $3$ и $4$, то эти же свойства выполнены и для графа $S''$.

\begin{proposition}
Для $S''$ выполнено свойство $5$.
\end{proposition}

\begin{proof}
Рассмотрим в графе $S'$ симметричный путь $f(s)$, построенный ранее (см. замечание \ref{N4_6}). Слово этого пути -- подслово сверхслова $W$.
Как доказывалось ранее, если ребру $d$ графа $S''$ в $S'$ соответствует путь $d_1d_2\dots d_k$, то $d_1d_2\dots d_k\sqsubseteq f(s)$.
Из следствия \ref{C2_3} имеем $F(d_1d_2\dots d_k)\sqsubseteq F(f(s))$ а также $B(d_1d_2\dots d_k)\sqsubseteq F(f(s))$. Осталось заметить, что $F(f(s))\sqsubseteq W$.
\end{proof}

\begin{proposition}
Для $S''$ выполнено свойство $7$.
\end{proposition}

\begin{proof}
Рассмотрим путь $f(s)$ (см. замечание \ref{N4_6}).
Докажем, что если $l$ -- симметричный путь в $s$ и $F(f(s))\sqsubseteq F(l)$, то для любого ребра $v\in S''$ выполняется $v\sqsubseteq l$.
В самом деле: путь $l$ является симметричным и в графе $S'$. По доказанному ранее свойству $4$ для схемы $S'$, путь $l$ содержит путь $f(s)$. А стало быть, в графе $S''$ он проходит по всем рёбрам.
\end{proof}

Для доказательства теоремы \ref{T1} осталось проверить свойство $6$. Пусть $u$ -- интересующее нас подслово сверхслова $W$.
Каждому ребру графа $S''$ в $S'$ соответствует путь, пусть все такие пути имеют длину не более $N$ рёбер, а слова всех рёбер имеют в $S$ длину не более $l_{\max}$.
Рассмотрим $w_0$ -- подслово в $W$ вида $u_1uu_2$, где $|u_1|=|u_2|=10\cdot N l_{max}$.
В $S$ существует допустимый путь $l$, слово которого содержит $w_0$. Рассмотрим в $S'$ путь $f(l)$. В этом пути содержится более $10\cdot N$ рёбер. Среди последних $N$ рёбер пути есть ребро, из конца которого выходит более одного хорошего ребра. Также среди первых $N$ рёбер пути есть ребро, в начало которого входят хотя бы два хороших ребра. Отрезок пути между этими двумя рёбрами обозначим $r$. Путь $r$ является симметричным, а его слово содержит слово $u$ в качестве подслова. Кроме того, $r$ является симметричным путём в схеме $S''$.

Итак, {\bf теорема \ref{T1} доказана}.

\section {Свойства слов, порождённых примитивными морфизмами.}

\begin{remark} \label{Rem6_1}
Мы можем считать, что примитивный морфизм $\varphi$ таков, что в $a_j\sqsubseteq \varphi(a_i)$ для любых букв $a_i$ и $a_j$ (иначе возьмём соответствующую степень этого морфизма).
\end{remark}

\begin{definition}
{\it Пословнная сложность}  $P(N)$ сверхслова $W$ -- количество различных подслов $W$ длины $N$.
\end{definition}

\begin{definition}
{\it Показатель рекуррентности} $P_2(N)$ для равномерно рекуррентного сверхслова $W$ -- минимальное число $P_2(N)$ такое, что в любом подслове сверхслова
$W$ длины $P_2(N)$ встретятся все подслова $W$ длины $N$. 
\end{definition}

\begin{lemma} \label{Lm6_4}
Для примитивного морфизма $\varphi$ существуют такие $C_1>0$, $C_2>0$, $\lambda_0>1$, что $C_1\lambda_0^k<|\phi^k(a_j)|<C_2\lambda_0^k$ для любых $k$ и буквы $a_j$.
\end{lemma}
\begin{proof}

Образы каждой буквы алфавита под действием морфизма $\varphi$ являются словами, содержащими весь алфавит (см. замечание \ref{Rem6_1}).
Пусть $\varphi(a_i)=A_i$. Рассмотрим матрицу $A$ размера $n\times n$ (где $n$ - мощность алфавита $\{a_i\}$) такую, что $A^i_j$ -- количество букв $a_j$ в $A_i$. Тогда $k$-й столбец матрицы $A^N$ является вектором количеств вхождений букв алфавита в $\varphi^N(a_k)$.

Все элемены матрицы $A$ положительны. По теореме Перрона-Фробениуса, у матрицы $A$ есть $\lambda_0$ -- положительное собственное значение, которое является собственным значением с строго максимальным модулем, причём у этого значения есть собственный вектор $v_0$ с положительными координатами.
Тогда для любого вектора $v$ существует константа $c(v)$ такая, что $A^k(v)=c(v) \lambda_0^k v + \bar o(k)$.
Константа $c(v)$ -- коэффициент при $v_0$ в разложении $v$ по жорданову базису.
При этом для любого ненулевого вектора с неотрицательными координатами $c(v)>0$. Таким образом, существуют такие положительные константы $C_1$ и $C_2$, что для любого $k$ и любой буквы $a_j$ верно двойное неравенство:
$C_1\lambda_0^k<|\phi^k(a_j)|<C_2\lambda_0^k$.
\end{proof}

\begin{lemma}      \label{Lm3}
Если $W$ -- чисто морфическое слово, порождённое примитивным морфизмом $\varphi$, то его показатель рекуррентности $P_2(N)$
ограничен сверху некой линейной фунуцией: $P_2(N)\leq C_3 N$.
\end{lemma}

\begin{proof}
Для любого $k$ сверхслово $W$ можно разбить на $\varphi^k(a_i)$ для различных $a_i$.
Если $u_1$ -- подслово сверхслова $W$ -- имеет длину не менее $2 C_2 \lambda_0^k$, то 
$u$ содержит образ $\varphi^k$ от некоторой буквы. Если $u_2$ -- подслово $W$ -- имеет длину не более $C_1 \lambda_0^k$, то $u_2$ содержится в образе $\varphi^k(a_ia_j)$ для некоторых букв $a_i$ и $a_j$, встречающихся подряд в $W$.

Существует такое натуральное число $m$, что любая пара букв, встречающаяся подряд в $W$, встретится подряд в $\varphi^m(a_1)$.
Следовательно, если для некоторого $k$ длина слова $u_1$ не более $C_1 \lambda_0^k$, а длина $u_2$ не менее $C_2 \lambda_0^{k+m+1}$, то $u_2$ содержит $u_1$, так как $u_2$ содержит $\varphi^{m+k}(a_1)$ и, следовательно, содержит образ $\varphi^k$ от любой встречающейся подряд в $W$ пары букв.

Если отношение длин слов $u_2$ и $u_1$ более $\lambda_0^{m+2}\frac{C_2}{C_1}$, то такое $k$ найдётся.
\end{proof}

\begin{lemma} \label{Lm6_5}
Пусть $W$ -- равномерно рекуррентное сверхслово над алфавитом $A$ с не более, чем линейным показателем рекуррентности $P_2(k)\leq C_3k$,
а $\psi\colon A^*\to A^*$ -- такой морфизм, что образом каждой буквы $a_i$ является либо эта же буква $a_i$, либо пустое слово.
Тогда, если слово $\psi(W)$ не является пустым, оно имеет не более, чем линейный показатель рекуррентности.
\end{lemma}
\begin{proof}
В $W$ должна быть буква, образ которой -- не пустое слово. Пусть $\psi(a_1)=a_1$. Существует такое $M$, что среди любых $M$ подряд идущих букв слова $W$
встречается $a_1$. Тогда всякое подслово $w$ сверхслова $\psi(W)$ является подсловом $\psi(u)$, где $u$ -- некоторое подслово $W$ длины не более $M|w|$.
Следовательно, любое подслово $W$ длины $P_2(M|w|)\leq C_3M|w|$ содержит $w$.
\end{proof}

\begin{lemma} \label{Lm6_6}
Пусть $W$ -- равномерно рекуррентное сверхслово над алфавитом $A$ с не более, чем линейным показателем рекуррентности $P_2(k)$, также путь $B$ -- конечный алфавит, а $\psi\colon A^*\to B^*$ --
нестирающий морфизм, продолжающийся на $W$. Тогда слово $\psi(W)$ является равномерно рекуррентным с не более, чем линейным показателем рекуррентности.
\end{lemma}
\begin{proof}
Пусть $|\psi(a_i)|\leq M$ для любой буквы $a_i$.
Разобьём $\psi(W)$ на образы $\psi(a_i)$. Любое подслово $u_1$ длины $k$ содержится в образе идущих подряд $k$ букв. Если длина $u_2$ -- подслова $\psi(W)$ -- не менее, чем $M(F_2(k)+1)$, то $u_2$ содержит образ $\psi$ от $F_2(k)$ идущих подряд букв сверхслова $W$, и, следовательно, содержит $u_1$.
\end{proof}

\begin{corollary} Пусть $W$ -- равномерно рекуррентное сверхслово над алфавитом $A$ с не более, чем линейным показателем рекуррентности,
$B$ -- конечный алфавит, а $\psi\colon A^*\to B^*$ -- произвольный морфизм, продолжающийся на $W$. Тогда $\psi(W)$ имеет не более чем линейный показатель рекуррентности.
\end{corollary}

\begin{lemma} \label{Lm6_7}
Если $W$ -- сверхслово с не более, чем линейным показателем рекуррентности, то пословная сложность $W$ также не более чем линейна.
\end{lemma}
\begin{proof}
Пусть $F_2(k)\leq C_3k$. Рассмотрим $u$ -- подслово $W$ длиной $C_3k$. В нём должны присутствовать все подслова $W$ длины $k$, но в $u$ таких подслов не более, чем $C_3k-k+1$. 
\end{proof}

\begin{lemma} \label{compl}
Если для сверхслова $W$ существует такая константа $C$, что для всех $N$ выполнено $P(N)\leq C N$, то функция $P(N+1)-P(N)$ ограничена. 
\end{lemma}

Доказательство приведено в работе \cite{Cassaigne}

\begin{lemma} \label{cas}
Если сверхслово $W=\psi(\varphi^{\infty}(a_1))$, где морфизм $\varphi$ примитивный, то существует такая константа $C_4$, что во всех графах Рози $G_k(W)$ количество вершин суммарной степени более $2$ не превосходит $C_4$.
\end{lemma}
\begin{proof}
К сверхслову $W$ применимы леммы \ref{Lm6_4} --- \ref{compl}.

Для каждого подслова $W$ длины $n$ в $W$ существует хотя бы одно подслово длины $n+1$ с тем же началом. Если же слову в графе $G_n$ соответствует вершина, у которой исходящая степень более одного, то в $W$ есть хотя бы два подслова с тем же началом.
Следовательно, для каждого $n$ число $P(n+1)-P(n)$ не менее числа вершин $G_n$ с исходящей степенью более единицы. Аналогично, $P(n+1)-P(n)$ не менее количества вершин входящей степени более единицы. Из того, что $P(n+1)-P(n)$ ограничено, следует утверждение леммы.
\end{proof}

\begin{lemma} \label{sdvig}
Если для равномерно рекуррентного непериодичного $W$ показатель рекуррентности $P_2(k)\leq Ck$, то существует такое $C_5>0$, что если $u\sqsubseteq W$,
то для любых двух различных вхождений $u$ в $W$ их левые концы находятся на расстоянии не меньшем, чем $C_5|u|$.
\end{lemma}
\begin{proof}
Пусть в $W$ есть два вхождения $u$ длины $k$, а их левые концы находятся на расстоянии $l<\frac{k}{2 C}$.
$$
\lefteqn{\overbrace{
\phantom{a_1\dots a_{l}a_{l+1}a_{l+2}\dots a_{k'-1}a_{k'}}}^{u}}
a_1\dots a_{l}\underbrace{a_{l+1}a_{l+2}\dots a_{k'-1}a_{k'}a_{k'+1}\dots a_{k'+l}}_{u}
$$
По очереди рассматривая оба вхождения слова $u$, приходим к выводу, что фрагмент $a_1\dots a_{l}$ повторится подряд как минимум $[\frac{k_1}{l}]$ раз, что превышает $C$. Любое подслово $W$ длины $l$ является подсловом слова $(a_1a_2\dots a_l)^C$, что немедленно влечёт периодичность $W$ и противоречие.
\end{proof}

\section{Свойства схем Рози для слов, порождённых примитивными морфизмами.}
Во всём разделе $W=\psi(\varphi^{\infty}(a_1))$, где $\psi$ -- произвольный морфизм, $\varphi$ -- примитивный морфизм, $W$ не является заключительно периодичным, а $S$ -- схема Рози для сверхслова $W$.

\begin{definition}
{\it Масштаб схемы} -- наименьшая из длин слов опорных рёбер.
\end{definition}

Возьмём собирающие вершины этой схемы. Для каждой из них возьмём ребро, выходящее из вершины, и его естественное продолжение вперёд. Получится набор симметричных путей $\{s_i\}$ схемы $S$. Рассмотрим раздающие вершины схемы $S$. Для каждой из них возьмём естественное продолжение назад ребра, входящего в вершину. Получим набор симметричных путей $\{t_i\}$.

\begin{lemma}      \label{Lm4}
Существует такое $C_6$, что для любых двух путей $s_1$, $s_2$ из $\{s_i\}$, длины слов $F(s_1)$ и $F(s_2)$ отличаются не более, чем в $C_6$ раз.
\end{lemma}
\begin{proof}
Так как $F(s_1)$ является передним словом первого ребра пути $s_1$, то $F(s_1)\sqsubseteq W$. Аналогично, $F(s_2)\sqsubseteq W$.
Ровно одно из рёбер пути $s_1$ входит в раздающую вершину. То же самое верно и для $s_2$.
Следовательно, случаи $s_1\sqsubseteq _2s_2$ или $s_2\sqsubseteq s_2s_1$ невозможны. Значит, по свойству $4$ схем Рози, в $F(s_1)$ может быть не более одного вхождения $F(s_2)$ и наоборот.
Для слова $W$ выполняется $P_2(n)\leq C_3n$ для некоторого $C_3$ (см. лемму \ref{Lm3}). Если бы длины слов $F(s_1)$ и $F(s_2)$ различались более, чем в $2C_3$ раз, то одно из слов (например, $F(s_1)$) можно было бы разбить на два слова, каждое из которых содержало бы $F(s_2)$.
\end{proof}

\begin{remark}
Также для любых путей из $\{t_i\}$ отношение длин их слов не превосходит $C_6$. Доказательство полностью аналогично.
\end{remark}

\begin{corollary} \label{7_4}
Если $M$ -- масштаб схемы $S$, то для любого пути из $\{s_i\}$ или $\{t_i\}$ длина его слова не превосходит $C_6M$.
\end{corollary}

\begin{lemma}      \label{Lm5}
Для любого пути из $\{s_i\}$, его слово является в $W$ специальным слева и
оканчивается на некоторое биспециальное слово длины не менее $M$, где $M$ -- масштаб схемы.
\end{lemma}
\begin{proof}
Пусть рассматриваемый путь имеет рёберную запись $v_1v_2\dots v_n$. Тогда ребро $v_n$ выходит из собирающей вершины, иначе естественное продолжение ребра $v_1$ было бы короче хотя бы на одно ребро. Ребро $v_n$ является опорным.

Докажем, что слово любого опорного ребра $v$ является биспециальным в $W$. Докажем, например, что $F(v)$ является специальным справа. Пусть из правого конца $v$ выходят $y_1$ и $y_2$. По свойству $7$ схем Рози для ребра $y_1$ есть $u_{y_1}$ -- подслово слова $W$ такое,
что любой симметричный путь, слово которого содержит $u_{y_1}$, проходит по ребру $y_1$. По лемме \ref{Lm2_4},
в схеме $S$ есть допустимый путь, слово которого содержит $u_{y_1}$. Рёберная запись этого допустимого пути обязана содержать $vy_1$.
Следовательно, в слово этого пути входит переднее слово пути $vy_1$, стало быть, $F(vy_1)\sqsubseteq W$.

Аналогично доказвается, что $F(vy_2)\sqsubseteq W$. Так как $F(vy_i)=F(v)F(y_i)$ и слова $F(y_i)$ имеют различные первые буквы по свойству $2$ схем Рози, $v$ является специальным справа. Специальность слева доказывается абсолютно так же.

Таким образом, $F(v_n)$ -- биспециальное. То, что $|F(v_n)|\geq M$, очевидно.

Теперь докажем, что слово пути $v_1v_2\dots v_n$ является специальным слева. Пусть $x_1$ и $x_2$ входят в начало ребра $v_1$. Для ребра $x_1$ есть $u_{x_1}$ -- такое подслово $W$, что любой симметричный путь, слово которого содержит $u_{x_1}$, содержит $x_1$.
В $S$ существует допустимый путь $l_{x_1}$, слово которого содержит $u_{x_1}$. Этот путь проходит по ребру $x_1$. Следующие $n$ рёбер этого пути могут быть только рёбра $v_1,v_2,\dots v_n$. Рассмотрим слово $F(l_{x_1})$: $B(x_1)F(v_1v_2\dots v_n)\sqsubseteq B(l_{x_1})$.
Значит, $B(x_1)F(s_1s_2\dots s_n)\sqsubseteq W$.

Аналогично, $B(x_2)F(v_1v_2\dots v_n)\sqsubseteq W$. Пользуясь свойством $2$ схем Рози для $B(x_1)$ и $B(x_2)$, получаем, что $F(v_1v_2\dots v_n)$ является специальным слева.
\end{proof}

Аналогично доказывается соответствующий факт для путей $\{t_i\}$.

\begin{lemma}      \label{fin}
Существует константа $C_7$ такая, что в любой схеме Рози для $W$ количество вершин не превосходит $C_7$.
\end{lemma}

\begin{proof}
Достаточно показать, что для любой схемы $S$ количество элементов в $\{s_i\}$ и $\{t_i\}$ ограничено константой, зависящей только от сверхслова $W$.
Докажем для $\{s_i\}$, так как для $\{t_i\}$ аналогично.

Пусть $M$ -- масштаб схемы $S$.
Рассмотрим для $W$ граф Рози $G_{k}$, где $k = [\frac M2]$. Пусть $s_1\in \{s_i\}$.
Из \ref{7_4}, его слово $F(s_1)$ имеет длину, не превосходящую $C_6M$.
Из леммы \ref{Lm5} следует, что последние $k$ букв этого слова являются словом, специальным справа, а первые $k$ букв -- словом, специальным слева.
При этом $F(s_1)$ является подсловом $W$. Все подслова в $F(s_1)$ длины $k+1$ образуют путь в $G_k$, ведущий из собирающей вершины в раздающую.
Пусть в этом пути обнаружился цикл длины $l$.

Это значит, что в $W$ есть подслово длины $k+l$, первые $k$ букв которого образуют то же слово, что и последние $k$ букв.
Применяя лемму \ref{sdvig}, получим $l\geq C_5k$.

Слову $f(s_1)$ соответствует путь по рёбрам графа $G_{k}$ из одной специальной вершины в другую,
этот путь не может попадать в одну и ту же вершину $G_k$ менее, чем через $C_5k$ шагов, а его длина не превосходит $C_6M$.
Значит, одну и ту же вершину графа $G_k$ этот путь посетит не более $\frac{C_6M}{C_5k}$ раз.

Из леммы \ref{cas}, в графе $G_k$ может быть не более $C_4(W)$ специальных вершин. Тогда посещений специальных вершин у пути будет не более
$\frac{C_6MC_4(W)}{C_5k}$.
Если $|B|$ - мощность нашего алфавита, то всего различных путей, содержащих не более $\frac{C_6M}{C_5k}$ (с учётом количества вхождений) специальных вершин, выходящих из специальной вершины и входящих в специальную вершину, будет не более, чем 
$$
\sum_{i=1}^{\frac{C_6MC_4(W)}{C_5k}}C_4(W) B^i.
$$

Для различных путей из $\{s_i\}$ их слова также различны, иначе, по свойству $4$ схем Рози, пути бы совпадали. Значит, этим путям мы поставили в соответствие различные пути в $G_k$. Поэтому верна оценка
$$
\#\{s_i\}\leq \sum_{i=1}^{\frac{C_6MC_4(W)}{C_5k}}C_4(W) B^i.
$$
\end{proof}

\begin{corollary} \label{sl7_7}
Количество различных облегчённых нумерованных схем, возникающих при детерменированной эволюции, конечно для $W$.
\end{corollary}

\begin{lemma} \label{Lm7_8}
Существует такая константа $C_8$, что если $S_1$ и $S_2$ -- нумерованные схемы Рози для $W$ и $S_2=\Evol(S_1)$,
то масштаб схем $S_1$ и $S_2$ относятся не более, чем в $C_8$ раз.
\end{lemma}
\begin{proof}
Для некоторого $C_3$ выполняется $P_2(k)\leq C_3k$.
Если $v$ -- опорное ребро схемы $S_2$, то ему соответствует симметричный путь $s$ по рёбрам схемы $S_1$, причём $F(s)=F(v_1)$.
Путь $s$ задаётся в $S_1$ номерами его рёбер.

Номера рёбер пути $s$ однозначно определяются по следующему набору: \{Облегчённая нумерованная схема $S_1$, множество пар плохих рёбер в $S_1$(задаваемых номерами этих рёбер), облегчённая нумерованная схема $S_2$, номер ребра $v$ в схеме $S_2$\}.

Из \ref{sl7_7} следует, что таких наборов для слова $W$ конечное число. Следовательно, для $W$ существует такая константа $C_9$, что путь $s$ не длиннее $C_9$ рёбер (для всех $S_1$, $S_2$ и $v$).
А значит, $F(s)\leq C_9C_3M$, где $M$ -- масштаб схемы $S_1$, иначе для некоторого опорного ребра $v_0$ графа $S_1$ будет
$F(v_0)\sqsubseteq _{C_9+1}F(s)$ и, по свойству $4$ схем Рози, $v_0\sqsubseteq _{C_9+1}s$.

Неравенство $|F(v)|=|F(s)|\geq M$ очевидно.
\end{proof}

\begin{lemma} \label{LmC_7}
Существует такая константа $C_{10}$, что для любой схемы $S$ длины всех слов на рёбрах этой схемы не превосходят $C_{10}M$, где $M$ -- масштаб схемы.
\end{lemma}
\begin{proof}
Достаточно доказать существования такой константы для передних слов, для задних доказательство не будет отличаться.
Для некоторого $C_3$ выполняется $P_2(k)\leq C_3k$.

Для рёбер, выходящих из собирающих вершин, существование такой константы следует из леммы \ref{Lm4}.
Рассмотрим произвольное ребро $v$, выходящее из раздающей вершины. По свойствам схем Рози, в $S$ есть допустимый путь $l_v$, проходящий через $v$.

Рассмотрим $l'_v$ -- минимальный симметричный подпуть этого пути, проходящий через $v$. Его рёберная запись имеет вид
$$
v_1v_2\dots v \dots v_n.
$$
Так как $l'_v$ -- минимальный, то среди рёбер $v_1v_2\dots v$ ровно одно выходит из собирающей вершины (а именно $v_1$). Также из минимальности следует, что среди рёбер $v\dots v_n$ ровно одно входит в раздающую вершину (а именно $v_n$).
Тогда $l'_v$ проходит не более, чем по двум опорным рёбрам.
А так как его слово является подсловом $W$, то из свойства $4$ для схем его длина $|F(l'_v)|$ не более, чем $2C_3M$.
Осталось заметить, что $F(v)\sqsubseteq F(l'_v)$.
\end{proof}

\begin{lemma} \label{Lmdlin}
Существует такая константа $C_{11}$, что если $S$ -- схема Рози для $W$, а $s_0$ -- допустимый путь, то $|F(s_0)|\geq C_{11}MN$, где $M$ -- масштаб схемы, а $N$ -- количество рёбер пути.
\end{lemma}
\begin{proof}
Пусть $v_1$ -- первое ребро произвольного допустимого пути $s$. Слово $F(v_1)$ является началом $F(s)$. Естественное продолжение вправо ребра $v_1$
-- симметричный путь, принадлежащий $\{s_i\}$.
По лемме \ref{Lm5} $F(s_1)$ является специальным слева, а его длина не меньше, чем $M$.
Аналогично доказывается, что $F(s)$ оканчивается на специальное справа слово длины не менее $M$.

Следовательно, слово $F(s)$ соответствует пути $l$ по рёбрам графа $G_{[\frac{M}{2}]}$, при этом $l$ начинается в левой специальной вершине, а кончается в правой специальной вершине.

Рассмотрим путь $s_0$. Среди его $N-1$ промежуточных вершин найдутся либо половина раздающих, либо половина собирающих.
Не умаляя общности, будем считать, что выполнено первое. Рассмотрим множество допустимых путей, являющихся началами $s_0$.
Их хотя бы $\frac{N-1}{2}$. Их можно упорядочить так, чтобы слово каждого следующего пути являлось началом слова предыдущего.
$$
F(s_0)\succeq F(s_1)\succeq \dots \succeq F(s_k), \:\:k\geq \frac{N-1}{2}.
$$

Словам $F(s_i)$ соответствуют пути $l_i$, начинающиеся в одной собирающей вершине, кончающиеся в раздающих вершинах и такие,
что каждый следующий путь -- начало предыдущего.
В $G_{[\frac{M}{2}]}$ не более $C_7$ специальных вершин (см. лемму \ref{cas}). Следовательно, через одну из раздающих вершин путь $l_0$
проходит не менее $\frac{N-1}{C_7}$ раз.

Из леммы \ref{sdvig} следует, что между последовательными посещениями одной и той же вершины графа $G_{[\frac{M}{2}]}$ в пути $l_0$ должно быть не менее $[\frac{M}{2}]C_5$ рёбер. Значит, $l_0$ проходит не менее, чем по $[\frac{M}{2}]C_5\frac{N-1}{C_7}$ рёбрам. Осталось сказать, что длина $l_0$
не более $|F(s_0)|$.
\end{proof}

\begin{lemma} \label{LmC_12}
Существует такая константа $C_{12}$, что для любой схемы $S$ и любого пути $s$ длина слова этого пути $F(s)$ не превосходит $C_{12}NM$, где $M$ -- масштаб схемы, а $N$ -- количество рёбер пути.
\end{lemma}
\begin{proof}
Если симметричный путь проходит по $N$ рёбрам, то длина его слова не превосходит $Nl_{max}$, где $l_{\max}$ -- максимальная из длин слов на рёбрах.
В силу леммы \ref{LmC_7}, $l_{max}\leq C_{10}M$. Лемма доказана.
\end{proof}

\section{Построение оснасток.}

В этом разделе $W=\psi(\varphi^{\infty}(a_1))$, где $\varphi$ -- примитивный морфизм, а $W$ не является заключительно периодичным.

\begin{definition}
Пусть $A\sqsubseteq W$, $S$ -- схема Рози для $W$. Множество симметричных путей в $S$, слова которых являются подсловами $A$, обозначим $S(A)$.
Если $S(A)$ непусто, в этом множестве есть максимальный элемент относительно сравнения $\sqsubseteq $, который будем называть {\it нерасширяемым путём}. Очевидно, для каждого пути из $S(A)$ есть путь, который содержит его и является нерасширяемым. Назовём этот путь {\it максимальным расширением}.
\end{definition}

\begin{lemma} \label{lm72}
Если $s$ -- нерасширяемый путь в $S(A)$, то $|A|-2l_{\max}\leq |F(s)|\leq |A|$, где $l_{\max}$ -- максимальная из длин слов на рёбрах $S$.
\end{lemma}
\begin{proof}
Неравенство $|F(s)|\leq |A|$ очевидно, так как $F(s)\sqsubseteq A$.

Пусть $A$ имеет вид $u_1F(s)u_2$. Если $|F(s)|<|A|-2l_{\max}$, то либо $|u_1|>l_{\max}$, либо $|u_2|>l_{\max}$. Не умаляя общности, предположим второе. По лемме \ref{Lmtochn}
существует такой допустимый путь $s_1$, что началом его является путь $s$, а его слово имеет вид $F(s_1)=F(s)u_2u_3$.
Пусть $s$ имеет рёберную запись $v_1v_2\dots v_n$, а $s_2$ -- рёберную запись $v_1v_2\dots v_nv_{n+1}\dots v_{n+k}$. Рассмотрим путь, образованный $s$ и естественным продолжением вправо ребра $v_{n+1}$. Слово этого пути является подсловом $A$, следовательно, $s$ не является нерасширяемым в $S(A)$.
\end{proof}

\begin{lemma} \label {Lmekviv}
Существует такая константа $C_{13}$, что для любой схемы Рози $S$ слова $W$, слов $X$, $Y$, $Z$ таких, что $XYZ\sqsubseteq W$ а также $\min\{|X|,|Y|,|Z|\}>C_{13}M$, где $M$ -- масштаб схемы $S$, и симметричного пути $l$ в схеме $S$ следующие два условия эквивалентны:
\begin{enumerate}
\item $l\in S(XYZ)$
\item Существует такой симметричный путь $l'$, что $l\sqsubseteq l'$, $l'$ разбивается на три части $l'=xyz$, где $xy$ -- нерасширяемый путь для $S(XY)$, $y$ -- нерасширяемый путь для $S(Y)$, $yz$ -- нерасширяемый путь для $S(YZ)$.
\end{enumerate}
\end{lemma}

\begin{proof}
Согласно лемме \ref{sdvig} существует такая константа $C_5$, что любое подслово сверхслова $W$ длины $n$ не может иметь два различных вхождения в подслово $W$ длины, не превосходящей $n+C_5n$.

Сначала докажем, что при достаточно больших $C_{13}$ $1$ влечёт $2$.

Будем доказывать, что в качестве $l'$ можно взять максимальное расширение $l$ в $S(XYZ)$.
По лемме \ref{lm72}, $F(l')\geq |XYZ|-2l_{\max}$, а из леммы \ref{LmC_7} следует $l_{\max}\leq C_{10}M$.

Рассмотрим в $l'$ такое максимальное начало $l_1$, что $l_1\in S(XY)$.
Докажем, что $l_1$ -- нерасширяемый в $S(XY)$. Пусть $F(l')=F(l_1)u_1$. Рассмотрим его вхождение в $XYZ$, как подслова $u$.

Так как $l'$ -- нерасширяемый в $XYZ$, то 
$$
XYZ=d_1F(l')d_2,
$$
где $|d_1|\leq l_{\max}$, $|d_1|\leq l_{\max}$.
$$
XYZ=d_1F(l_1)u_1d_2.
$$

Есть два случая:
\begin{enumerate}
\item $|d_1F(l_1)|\leq |XY|$. В этом случае получаем $|d_1F(l_1)|\geq  |XY|-l_{\max}$, иначе у $l'$ есть начало $l'_1$ такое, что $l_1\sqsubseteq l'_1$ и
$|F(l'_1)|\leq |F(l_1)|+l_{\max}$, а стало быть, $F(l'_1)\sqsubseteq XY$.

Тогда $|F(l_1)|\geq |XY|-2l_{\max}$.
Предположим, что $l_1$ расширяемый в $XY$. Тогда для некоторого $l''_1$ выполняется $l_1\sqsubseteq l''_1$ и $F(l''_1)\sqsubseteq XY$.
Можно считать, что $l_1$ -- конец или начало $l''_1$.
	\begin{enumerate}
	\item $l''_1=l_bl_1$. Тогда $XY=e_1B(l_b)F(l_1)e_2$. Если выполнено
	$$
	2C_{10}M<C_5(2C_{13}M-2C_{10}M),
	$$
	то у $F(l_1)$ только одно вхождение в $XY$, то есть $d_1=e_1B(l_b)$ и 
	$$
	XYZ=e_1B(l_b)F(l')d_2.
	$$
	Тогда симметричный путь $l_bl'\in S(XYZ)$ и путь $l'$ расширяемый.
	\item $l''_1=l_1l_e$. Тогда $XY=e_1F(l_1)F(l_e)e_2$. Если выполнено
	$$
	2C_{10}M<C_5(2C_{13}M-2C_{10}M),
	$$
	то у $F(l_1)$ только одно вхождение в $XY$, то есть $e_1=d_1$ и $$XY=d_1F(l_1)F(l_e)e_2.$$
	
	Следовательно, $F(l_1l_e)\sqsubseteq XY$.
	Кроме того, $F(l_1l_e)$ является началом $F(l')$. Тогда, по $2$-му свойству схем Рози, $l_1l_e$ является началом $l'$. Противоречие.
	\end{enumerate}

\item $|d_1F(l_1)|>|XY|$. В этом случае $d_1F(l_1)=XYd_3$, где $|d_3|<|d_1|\leq l_{\max}$.
Так как
$F(l_1)\sqsubseteq _2XYd_3\sqsubseteq W$, $|XYd_3|<|XY|+C_{10}M$, $|F(l_1)|>|XY|-C_{10}M$, то при 
$$
2C_{10}M<C_5(2C_{13}M-C_{10}M)
$$
получаем противоречие.
\end{enumerate}

Аналогично, если $l_2$ -- такой максимальный конец пути $l'$, что $l_2\in S(YZ)$, то $l_2$ -- нерасширяемый в $S(YZ)$.

Докажем несложный факт: пути $l_1$ и $l_2$ внутри $l'$ перекрываются, то есть в рёберных записях $l_1$ и $l_2$ суммарно больше рёбер,
чем в рёберной записи пути $l'$. В самом деле, иначе возьмём $s$ -- естественное продолжение влево последнего ребра $l_1$.
Тогда $|F(l')|+|F(s)|\geq F(l_1)+F(l_2)$, то есть $$|XYZ|\geq (|XY|-2C_{10}M)+(|YZ|-2C_{10}M)-C_{10}M,$$
что не выполняется при достаточно больших $C_{13}$.


Итак, можно считать, что $l_1=xy$, $l_2=yz$, $l'=xyz$. Для доказательства импликации осталось показать, что путь $y$ -- нерасширяемый путь в $S(Y)$.

Имеем $XYZ=d_1B(x)F(y)F(z)d_2$, $XY=d_1B(x)F(y)e_1$ и $YZ=e_2F(y)F(z)d_2$, где $\max\{|d_1|,|d_2|,|e_1|,|e_2|\}\leq l_{\max}$.
Отсюда 
$$
|Y|=|e_1|+|e_2|+|F(y)|,
$$
то есть $Y=e_2F(y)e_1$.

Если $y$ расширяемый в $S(Y)$, то, не умаляя общности, $y$ -- начало пути $yl_e$ такого, что $F(y)F(l_e)\sqsubseteq Y$, то есть $Y=f_1F(y)F(l_e)f_2$.
Если выполняется неравенство
$$
2MC_{10}<C_5(C_{13}M-2MC_{10}),
$$
то $F(y)$ имеет ровно одно вхождение в $Y$, стало быть, $e_2=f_1$ и $e_1=F(l_e)f_2$.

Тогда $XY=d_1B(x)F(y)F(l_e)f_2$, то есть $l_1l_e\sqsubseteq S(XY)$. Кроме того, $F(l_1l_e)$ является началом слова $F(l_1)e_1=B(x)F(y)e_1$, которое,
в свою очередь, является началом $B(x)F(Y)F(z)=F(l')$. Следовательно, $l_1l_e$ -- начало $l'$. Противоречие.

Теперь докажем, что при достаточно больших $C_{13}$ $2$ влечёт $1$.

Так как $xy$ -- нерасширяемый путь в $S(XY)$, а $y$ -- нерасширяемый путь в $Y$, $XY$ имеет вид $XY=d_1B(x)F(y)d_2$, а $Y$ имеет вид $Y=e_1F(y)e_2$, причём $|d_i|,|e_i|\leq l_{\max}$.
Если $C_{10}M<C_5(C_{13}M-C_{10}M)$, то $d_2=e_2$.

Аналогично, $YZ=f_1F(y)F(z)f_2$, где $f_1=e_1$. Тогда $XYZ=d_1B(x)F(y)F(z)d_2$, а значит, $F(l)\sqsubseteq F(l')=F(xyz)\sqsubseteq XYZ$.
\end{proof}

\begin{definition} \label{d7_4}
{\it Набор проверочных слов порядка $k$} -- это набор слов $\{q_i\}$, в который входят $\psi(\varphi^k(a_i))$ для всех букв алфавита $\{a_i\}$, а также
$\psi(\varphi^k(a_ia_j))$ для всевозможных пар последовательных букв слова $\varphi^{\infty}(a_1)$.
\end{definition}

Кроме того, для любого $k$ в наборе проверочных слов порядка $k$ одно и то же число подслов.

\begin{proposition} \label{pr7_5}
Существуют такие положительные $C_{14}$, $C_{15}$ и $\lambda_0>1$, что в наборе проверочных слов с номером $k$ длина любого слова $q$ удовлетворяет
двойному неравенству $C_{14}\lambda_0^k<|q|<C_{15}\lambda_0^k$.
\end{proposition}
\begin{proof}
Существуют константы $k_1$ и $k_2$ такие, что, с одной стороны, в $\varphi^k(a_i)$ среди любых $k_1$ идущих подряд букв есть буква, которая не удаляется при действии $\psi$, а с другой стороны, образ $\psi$ от каждой буквы -- это слово не длиннее $k_2$. Осталось применить лемму \ref{Lm6_4}. 
\end{proof}

\begin{definition}
{\it Оснастка} $(S,k)$ порядка $k$ определяется для пронумерованной схемы Рози $S$.
Для получения оснастки берётся набор проверочных слов порядка $k$, для каждого проверочного слова $v_i$ рассматривается набор $S(q_i)$. Каждый путь в схеме задаётся упорядоченным набором чисел -- номеров рёбер схемы. Таким образом, $S(q_i)$ задаётся множеством таких упорядоченных наборов, а оснастка получается, если взять такие наборы для всех проверочных слов и облегчённую нумерованную схему $S$.
\end{definition}

\begin{definition}
{\it Размер оснастки} -- максимальная длина (в рёбрах) по всем путям из $S(q_i)$ для всех $q_i$ -- проверочных слов порядка $k$.
\end{definition}

\begin{lemma} \label{size}
Существуют такие положительные $C_{16}$, $C_{17}$ и $C_{18}$, что для любой схемы $S$ размер оснастки $(S,k)$ заключён между
$C_{16}\frac{\lambda_0^k}{M}-C_{17}$ и $C_{18}\frac{\lambda_0^k}{M}$.
\end{lemma}
\begin{proof}
Длина каждого проверочного слова, согласно \ref{pr7_5}, хотя бы $C_{14}\lambda_0^k$. Если $q_i$ -- проверочное, то длина слова нерасширяемого в
$S(q_i)$ пути не менее $C_{14}\lambda_0^k -- 2C_{10}M$. А длина этого пути в рёбрах составляет, согласно лемме \ref{lm72}, не менее $\frac{C_{14}\lambda_0^k - 2C_{10}M}{C_{12}M}$.

С другой стороны, длина слова нерасширяемого пути не может быть больше $C_{15}\lambda_0^k$, а его длина в рёбрах, как следует из леммы \ref{Lmdlin}, не может быть более $\frac{C_{15}\lambda_0^k}{C_{11}M}$.
\end{proof}

\begin{corollary} \label{wts}
Существуют такие $C_{19}$, $C_{20}$ и $C_{21}$, что если размер оснастки $(S,k)$ больше $C_{19}$, то
\begin{enumerate}
\item Длины всех проверочных слов составляют не менее $C_{13}M$.
\item Для любой хорошей тройки рёбер существует проходящий через неё путь, слово которого содержится во всех проверочных словах.
\item Размер оснастки $(S,k)$ относится к размеру оснастки $(\Evol(S),k)$ не более, чем в $C_{20}$ раз.
\item Размеры оснасток $(S,k)$ и $(S,k+1)$ отличаются не более, чем в $C_{20}$ раз.
\item Размер оснастки $(S,k+C_{21})$ больше размера оснастки $(S,k)$ хотя бы в два раза.
\end{enumerate}
\end{corollary}
\begin{proof}

\begin{enumerate}
\item Из лемм \ref{size} и \ref{pr7_5} следует, что если $x$ -- размер оснастки, а $|q|$ -- размер проверочного слова $q$, то
$|q|>C_{14}\lambda_0^k>C_{14}(C_{18}xM)$.
\item Симметричный путь, получающийся естественным расширением хорошей тройки рёбер сначала вправо, а потом влево,
является допустимым и его слово имеет длину не более $3l_{\max}\leq 3C_{10}M$.
Если проверочное слово $q$ имеет длину хотя бы $3C_{10}MC_5$, где $C_5$ -- показатель рекуррентности, то подсловом $q$ является и слово рассматриваемого пути. Как видно из предыдущего пункта, выбором $C_{19}$ этого легко добиться.
\item Пусть $M$ и $M_1$ -- масштабы схемы $S$ и $\Evol(S)$, а $x$ и $x_1$ -- размеры оснасток. Согласно лемме \ref{Lm7_8}, $M_1\leq C_8M$.
Тогда $x_1>C_{16}\frac{\lambda_0^k}{C_8M}-C_{17}>C_{16}\frac{C_{18}x}{C_8}-C_{17}$, что при достаточно большом $x$ больше, чем $C_{16}\frac{C_{18}x}{2C_8}$.
\item Пусть $x$ и $x_1$ -- размеры оснасток $(S,k)$ и $(S,k+1)$.

Тогда $C_{16}\frac{\lambda_0^{k+1}}{M}-C_{17}<x_1<C_{18}\frac{\lambda_0^{k+1}}{M}$. С другой стороны, $\frac{x+C_{17}}{C_{16}}>\frac{\lambda_0^{k}}{M}>\frac{x}{C_{18}}$.

Таким образом, $C_{16}\frac{\lambda_0x}{C_{18}M}-C_{17}<x_1<C_{18}\frac{\lambda_0(x+C_{17})}{C_{16}M}$.

\item Пусть $x$ и $x_1$ -- размеры оснасток $(S,k)$ и $(S,k+C_{21})$.

Тогда $C_{16}\frac{\lambda_0^{k+C_{21}}}{M}-C_{17}<x_1$. С другой стороны, $\frac{\lambda_0^{k}}{M}>\frac{x}{C_{18}}$.

Таким образом, $x_1>C_{16}\frac{\lambda_0^{C_{21}}x}{C_{18}M}-C_{17}$, что не меньше $2x$ при достаточно большом $C_{21}$.
\end{enumerate}
\end{proof}

\begin{lemma} \label{lm79}
Если размер оснастки $(S,k)$ хотя бы $C_{19}$, то оснастка $(S,k+1)$ является функцией оснастки $(S,k)$.
\end{lemma}
\begin{proof}
Рассмотрим проверочное слово порядка $k+1$.

Оно имеет вид $\psi (\varphi ^k(\varphi (a)))$, где $a$ -- либо одна, либо две буквы. В любом случае, $\varphi (a)$ - подслово $\varphi ^{\infty}$. Таким образом, проверочное слово порядка $k+1$ разбивается на блоки -- проверочные слова порядка $k$, причём любые два последовательных блока образуют проверочное слово порядка $k$. (заметим, что то, на какие типы слов порядка $k$ разбивается $\psi (\varphi ^{k+1}(a))$, не зависит от номера $k$. {\it Типом слова} $\psi (\varphi ^k(a_i))$ называется буква $a_i$.)

Пусть $A_1A_2\dots A_m$ -- проверочное слово порядка $k+1$, $A_i$ -- проверочные слова порядка $k$. 
Согласно \ref{wts}, длины слов $A_i$ удовлетворяют условию леммы \ref{Lmekviv}. Применим эту лемму к словам $A_1$,$A_2$,$A_3$. Так как свойство пути быть {\it нерасширяемым} определяется по оснастке, то по оснастке можно определить, какие пути принадлежат $S(A_1A_2A_3)$ (а точнее, какие пути в соответствующей облегчённой пронумерованной схеме). После этого применим лемму \ref{Lmekviv} к словам $A_1A_2$, $A_3$, $A_4$ и определим, какие пути принадлежат $S(A_1A_2A_3A_4)$. Потом определим, какие пути принадлежат $S(A_1A_2A_3A_4A_5)$, и так далее, наконец определим, какие пути схемы $S$ принадлежат $S(A_1A_2\dots A_m)$.
Делая так для всех проверочных слов порядка $k+1$, получим что хотели.

\end{proof}

\begin{lemma} \label{Lmvtr}
Если размер оснастки $(S,k)$ не менее $C_{19}$, то по оснастке $(S,k)$ однозначно определяются плохие пары рёбер и оснастка $(\Evol(S),k)$.
\end{lemma}

\begin{proof}

Из следствия \ref{wts} очевидно, что по оснастке определяется множество хороших и плохих тройки рёбер: по хорошим тройкам рёбер пути из оснастки проходят, а по плохим -- нет (так как пути в оснастках допустимые).
Следовательно, определяется и нумерованная облегчённая схема $\Evol(S)$. Симметричные пути в $\Evol(S)$ соответствуют некоторым симметричным путям в $S$ с теми же словами. При этом мы можем указать соответствие, глядя лишь на облегчённые схемы.
Таким образом, для каждого симметричного пути в облегчённой нумерованной схеме $\Evol(S)$ по оснастке $(S,k)$ можно определить, каким из $\Evol(S)(A_i)$ этот путь принадлежит. (Здесь $A_i$ -- проверочные слова порядка $k$).
\end{proof}

Симметричным путям в $\Evol(S)$ соответствуют симметричные пути в $S$ с теми же словами, при этом в $\Evol(S)$ пути не длиннее, чем 
соответственные пути в $S$. Стало быть, размер оснастки $(\Evol(S),k)$ не более, чем размер оснастки $(S,k)$.

{\bf Завершение доказательства теоремы \ref{Th1}.}

Построим последовательность оснасток по следующему правилу. Возьмём $T=2 C_{19} C_{20}^{C_{21}}$.
Построим первую оснастку для такого порядка проверочных слов, чтобы размер оснастки был хотя бы $T$. Каждая следующая оснастка определяется по предыдущей. Если её размер менее $T$, то по оснастке вида $(S,k)$ строится оснастка вида $(S,k+1)$.  Иначе делается операция перехода от оснастки вида $(S,k)$ к оснастке вида $((\Evol(S)),k)$. Таким образом, размер оснастки никогда не упадёт ниже $C_{19}$, значит, каждый шаг определён однозначно. С другой стороны, размер оснастки не поднимается выше $C_{20}T$. Значит, различных оснасток в последовательности конечное число (так как различных нумерованных облегчённых схем тоже конечно.) То есть с некоторого момента оснастки повторяются периодично. При этом не может быть так, что на каждом шаге повышается порядок проверочных слов. Значит, время от времени в последовательности появляются оснастки с новыми схемами. А значит, и протокол эволюции с некоторого места периодичен.

\section{Переход к произвольным равномерно рекуррентным подстановочным словам}
Пусть $\varphi$ -- произвольный морфизм из $A^*$ в $A^*$, $h$ -- морфизм из $A^*$ в $B^*$.

\begin{lemma} \label{perehod}
Если непериодичное слово $W=h(\varphi ^{\infty}(a))$ является равномерно рекуррентным, 
то для некоторого алфавита $D$ существуют такие морфизмы $\rho :D^* \to D^*$ и $g:D^* \to B^*$, где $\rho$ -- примитивный морфизм,
что множества подслов $W$ и $g(\rho ^{\infty}(d))$ для некоторой буквы $d\in D$ совпадают.
\end{lemma}

Для доказательства этой леммы нам потребуются следующие определения. Слово $w\in A^*$ будем называть $\varphi-${\it ограниченным}, если последовательность
$$w,\varphi(w),\varphi^2(w),\varphi^3(w),\dots$$
 периодична начиная с некоторого момента.
В противном случае, $|\varphi^n(w)|\rightarrow \infty$ при $n\rightarrow \infty$ и слово $w$ называется {\it $\varphi-$растущим}.
Очевидно, слово является $\varphi-$ограниченным тогда и только тогда, когда оно состоит из $\varphi-$ограниченных букв.

Назовём слово $w$ {\it $\varphi-h-$уничтожаемым}, если $h(\varphi^n(w))=\Lambda$ для любого $n\geq 0$. Легко видеть, что слово является $\varphi-h-$уничтожаемым тогда и только тогда, когда оно состоит из $\varphi-h-$уничтожаемых букв.
Далее, можно считать, что в алфавите $A$ нет $\varphi-h-$уничтожаемых букв:

\begin{proposition} \label{prerase}
Пусть $A'$ -- множество не $\varphi-h-$уничтожаемых букв.
Рассмотрим морфизм $\varphi ':A'^* \to A'^*$, определяемый как
$\varphi'(a_i)=$ "$\varphi(a_i)$ с вычеркнутыми $\varphi-h-$уничтожаемыми буквами".
Тогда слова $h(\varphi ^{\infty}(a))$ и
$h(\varphi'^{\infty}(a))$ совпадают.
\end{proposition}

Доказательство очевидно.

\begin{lemma}[Ehrenfeucht, Rozenberg] \label{eh}
Следующие три условия эквивалентны:
\begin{enumerate}
\item Слово $\varphi ^{\infty}(a)$ не является равномерно рекуррентным.
\item В $\varphi ^{\infty}(a)$ есть бесконечно много $\varphi -$ограниченных подслов.
\item Существует непустое $w\in A^*$ такое, что $w^n$ является подсловом $\varphi ^{\infty}(a)$ для любого $n$.
\end{enumerate}
\end{lemma}
Доказательство импликации $1\rightarrow 2$ можно найти в работе \cite{Pr}, импликации $2\rightarrow 1$ -- в работе \cite{Ehren},
а импликация $3\rightarrow 1$ очевидна.

Предположим, что слово $\varphi ^{\infty}(a)$ не является равномерно рекуррентным. Тогда некоторое $w$ встречается в нём сколь угодно много раз подряд.
Так как $w$ не является $\varphi-h$-уничтожаемым (см. \ref{prerase}), то для некоторого $k$ слово $h(\varphi^k(w))$ не является пустым словом.
Тогда слово $h(\varphi^k(w))$ встречается в $W$ сколь угодно много раз подряд, и из равномерной рекуррентности $W$ следует его периодичность.

Таким образом, мы можем считать, что $\varphi ^{\infty}(a)$ равномерно рекуррентно и, следовательно, в нём лишь конечное число $\varphi-$ограниченных
слов.

Пусть $I_{\varphi}$ -- множество всех $\varphi-$растущих букв, $B_{\varphi}$ -- множество $\varphi-$ограниченных слов (включая пустое).
В силу \ref{eh}, $B_{\varphi}$ конечно.
Рассмотрим (конечный) алфавит $C$, состоящий из символов $[twt']$, где $t$ и $t'$ буквы из $I_{\varphi}$, а 
$w$ -- слово из $B_{\varphi}$ и слово $twt'$ является подсловом $\varphi^{\infty}(a)$.

Определим морфизм $\psi:C^* \to C^*$ следующим образом:
$$
\psi([twt'])=[t_1wt_2][t_2wt_3]\dots[t_kw_kt_{k+1}],
$$
где $\varphi(tw)=w_0t_1w_1t_2\dots t_kw'_k$, слово $\varphi(t')$ начинается с $w''_kt_{k+1}$ и $w_k=w'_kw''_k$ (cлова $w_i$, $w'_k$ и $w''_k$ принадлежат $B_{\varphi} $).

Также определим $f:C^*\to C^*$ по правилу
$$
f([twt'])=h(tw).
$$

\begin{proposition} \label{mor_hor}
 Все буквы алфавита $C$ являются $\psi-$растущими и ни одна не является $\psi-f-$уничтожаемой.
\end{proposition}

\begin{proof}
Очевидно, в образе $\varphi(t)$ от произвольной буквы $t\in I_{\varphi}$ содержится хотя бы одна буква из $I_{\varphi}$. Более того, в слове $\varphi ^n(t)$ для некоторого $n$ содержатся хотя бы две буквы из $I_{\varphi}$, иначе $\varphi ^n(t)=w_nt_{i_n}v_n$, где $w_n$ и $v_n$ принадлежат $B_{\varphi }$.
Таких слов, согласно \ref{eh}, конечное число. Стало быть, буква $t$ не является $\varphi -$растущей.

Далее, если в $\varphi^n(t)$ встречается $k$ букв из $I_{\varphi }$, то длина $|\psi ^n([twt'])|$ не меньше, чем $k$.
Следовательно, буква $[twt']$ алфавита $C$ является $\psi-$растущей.

Предположим, что буква $[twt']$ является $\psi-h-$уничтожаемой. Тогда в некоторая степень морфизма $\psi$ от этой буквы является словом из $C^*$
длины не менее $2$. Значит, для некоторых $t_i$,$w_i$ слово вида $[t_1w_1t_2][t_2w_2t_3]$ является $\psi-f-$уничтожаемым. 
Буква $t_2$ не является $\varphi-h-$уничтожаемой. Значит, существует последовательность букв $\{a_i\}$ такая,
что $t_2=a_0$, $a_{i+1}$ содержится в $\varphi(a_i)$ и $h(a_k)\not =\Lambda$.

Докажем по индукции следующий факт: $\psi^n([t_1w_1t_2][t_2w_2t_3])$ содержит подслово вида $[t_{i_n}w_{i_n}t_{j_n}][t_{j_n}w_{j_n}t_{k_n}]$
такое, что в $w_{n_1}$, $t_{n_2}$ или в $w_{n_2}$ содержится $a_n$. пусть это верно для $n$, докажем для $n+1$. Возможны следующие случаи:
\begin{enumerate}
\item Пусть $a_n$ содержится в $w_{i_n}$. Тогда $\varphi(t_{i_n}w_{i_n})$ содержит подслово вида $t_{i_{n+1}}w_{i_{n+1}}$, где $a_{n+1}$
содержится в $w_{i_{n+1}}$. Тогда $\psi([t_{i_n}w_{i_n}t_{j_n}])$ содержит букву вида $[t_{i_{n+1}}w_{i_{n+1}}t_{j_{n+1}}]$, где $w_{i_{n+1}}$
содержит $a_{n+1}$. Так как у буквы $[t_{i_{n+1}}w_{i_{n+1}}t_{j_{n+1}}]$ в 
$\psi^{n+1}([t_1w_1t_2][t_2w_2t_3])$ есть левый либо правый сосед, искомое подслово встретится.
\item Если $a_n$ содержится в $w_{j_n}$, то доказательство полностью аналогично.
\item Пусть $a_n=t_{j_n}$, $\varphi(t_{j_n})=w_0t_1w_1t_2\dots$ Если $a_{n+1}$ встречается в $w_0$, то $a_{n+1}$ содержится в ``средней части'' последней буквы слова
$\psi([t_{i_n}w_{i_n}t_{j_n}])$. Иначе $a_{n+1}$ содержится как ``первая буква'' или в ``средней части'' одной из букв слова $[t_{j_n}w_{j_n}t_{k_n}]$.
\end{enumerate}
Таким образом, в  $\psi^k([t_1w_1t_2][t_2w_2t_3])$ найдётся буква, которая не обнуляется морфизмом $f$.
\end{proof}

Пусть $\varphi^{\infty}(a)$ имеет вид $a_1w_1a_2\dots$, где $a_1,a_2\in I_{\varphi}$, $w_1\in B_{\varphi}$. Тогда, как легко заметить, для любого $n$ слово $h(\varphi ^n(a))$ является началом слова $f(\psi^n([a_1w_1a_2]))$, следовательно, слова $W$ и $f(\psi^{\infty}([a_1w_1a_2]))$ совпадают.

\begin{remark} Конструкция морфизма $\psi$ была рассмотрена в работе \cite{Pr}.
\end{remark}

Рассмотрим ориентированный граф $G_{\varphi}$, вершинами которого являются буквы алфавита $C$, и из $c_i$ ведёт стрелка в $c_j$ тогда и только тогда, когда $c_j$ содержится в $c_i$. Пусть $D$ -- сильносвязная компонента этого графа, до которой можно дойти по стрелочкам из $[a_1w_1a_2]$. Тогда $\psi$ можно рассматривать как морфизм из $D^*$ в $D^*$. Существует буква $d\in D$ и натуральное $k$ такое, что $\psi^k(d)$ начинается с $d$. Обозначим $\rho=\psi(k)$. Слово $\rho^{\infty}(d)$ является словом, все конечные подслова которого являются подсловами $\psi^{\infty}([a_1w_1a_2])$.
Стало быть, все конечные подслова $f(\rho^{\infty}(d))$ являются подсловами $W$. Очевидно, что $\rho$ является примитивным.

При этом образ при $f$ хотя бы одной буквы из $D$ не равен пустому слову. Стало быть, $f(\rho^{\infty}(d))$ бесконечное. А так как $W$ равномерно рекуррентное, то все его конечные подслова являются конечными подсловами $f(\rho^{\infty}(d))$. В самом деле, пусть $w$ -- подслово $W$. Тогда для некоторого $n$ любое подслово $W$ длины $n$ содержит $w$. Но в $W$ есть такое подслово, являющееся подсловом $f(\rho^{\infty}(d))$.

Это рассуждение завершает доказательство леммы \ref{perehod}.

\begin{corollary} Пусть $W$ -- морфическое равномерно рекуррентное бесконечное слово. Тогда эволюция схем Рози для этого слова периодична с предпериодом.
\end{corollary}

Это утверждение вытекает из теоремы \ref{Th1}, леммы \ref{perehod} и того соображения, что эволюция схем Рози однозначно задаётся множеством подслов сверхслова. 

\section{Равномерно-рекуррентные слова с периодичной эволюцией}
Пусть у равномерно-рекуррентного сверхслова $W$ эволюция  схем Рози, начиная с некоторого места, периодична. Докажем, что его язык подстановочен.

В самом деле, пусть на некотором шаге эволюции на рёбрах схемы Рози написаны слова $A_1, A_2,\dots, A_n$. После одного шага элементарной эволюции на рёбрах следующей схемы Рози будут написаны новые слова $B_1,B_2,\dots,B_m$, причём каждое из слов $B_i$ является мономом от слов вида $A_j$ (то есть является конкатенацией нескольких слов, возможно, одного). Очевидно, множество мономов (то есть то, какие слова из $B$ являются конкатенацией каких и в каком порядке слов из $A$) определяется только куском протокола эволюции на соответствующем шаге.

Пусть протокол периодичен начиная с $N-$го члена и длина периода равна $k$. Рассмотрим некоторое $n_0>N$. Если рассматривать набор слов на рёбрах схемы Рози с номером $n_0$, будет упорядоченное множество $A_1,A_2,\dots,A_M$ из $M$ слов. Если взять схему с номером $n_0+k$, то множество слов на её рёбрах будет также иметь мощность $M$. Обозначим это множество $B_1,B_2,\dots,B_M$. Каждое слово из этого набора будет неким мономом от слов вида $A_i$. Обозначим эти мономы $f_1,f_2,\dots,f_M$. Имеем $B_i=f_i(A_1,A_2,\dots,A_M)$. Аналогично, на рёбрах схемы с номером $n+2k$ написано множество слов
$f_1(B_1,B_2,\dots,B_M),f_2(B_1,B_2,\dots,B_M),\dots,f_M(B_1,B_2,\dots,B_M)$ для того же множества мономов $\{f_i\}$ и так далее.
Рассмотрим алфавит из $M$ букв: $a_1,a_2,\dots,a_M$ и два морфизма. Первый морфизм $\varphi(a_i)=f_i(a_1,a_2,\dots,a_M)$, а второй -- $h(a_i)=A_i$.

Легко видеть, что множество слов на рёбрах схемы с номером $n+kl$ -- это в точности $h(\varphi^l(a_1)),h(\varphi^l(a_2)),\dots,h(\varphi^l(a_M))$. Так как длина слов $h(\varphi^l(a_1))$ при увеличении $l$ стремится к бесконечности, все эти слова являются подсловами $W$, а само $W$ -- равномерно-рекуррентное, то множество подслов $W$ является множеством подслов $h(\varphi^{\infty}(a_1))$, что и требовалось доказать.


\begin{thebibliography}{aaa}

\bibitem{Pr}  Francois Nicolas, Yuri Pritykin. On uniformly recurrent morphic sequences//
International Journal of Foundations of Computer Science, Vol. 20, No. 5 (2009) 919--940

\bibitem{Ehren} A. Ehrenfeucht and G. Rozenberg. Repetition of subwords in DOL languages Information
and Control, 59(1--3):13--35, 1983.

\bibitem{Frid}  А. Э. Фрид, О графах подслов $DOL$-последовательностей// Дискретн. анализ и исслед. опер., сер. 1, 6:4 (1999), 92Ц103

\bibitem{PS} Ан. А. Мучник, Ю. Л. Притыкин, А. Л. Семенов, "Последовательности,
близкие к периодическим"// УМН, 64:5(389) (2009), 21-96

\bibitem{Arn} В. И. Арнольд, Малые знаменатели и проблемы устойчивости движения в классической и небесной
механике// Успехи Мат. Наук, 1963, 18:6(114), стр. 191--192

\bibitem{Babenko}
  {\sl Бабенко И.~К.} {\it
   Проблемы роста и рациональности в алгебре и топологии.}
  Успехи мат. наук,
  1986, vol 41, no 2, pages 96--142.

\bibitem{BBL} А.Я.Белов, В.В.Борисенко, В.Н.Латышев, Мономиальные
алгебры // Итоги науки и техники. Совр. Мат. Прил. Тем. Обзоры т.
26 (алг. 4), М. 2002. 35-214.


\bibitem{BK} А.Белов, Г.Кондаков, Обратные задачи символической
динамики // Фундаментальная и прикладная математика, Т1, |1, 1995


\bibitem {KS} А. Б. Каток, А. М. Степин, Аппроксимации в эргодической теории // Успехи Мат.Наук, 1967, 22:5(137), 81Ц106



\bibitem{KudriavAleshPodkolzin} В.~Б.~Кудрявцев, С.~В.~Алешин, А.~С.~Подколзин, Введение в теорию автоматов // Москва, ``Наука'', 1985. 320 стр.


\bibitem{Os} В.И.Оселедец, О спектре эргодических автоморфизмов // ДАН СССР, 1966, 168, 1009Ц1011.


\bibitem{Salomaa}
Саломаа А. Жемчужины формальных языков. ---
 М.: Мир, 1986,  159 стр.


\bibitem{Sin} Я.Г. Синай, Введение в эргодическую теорию // М.: ФАЗИС, 1996. 144 с.



\bibitem{Uf7} Уфнаровский В.А. Комбинат. и асимпт. методы в алгебре. // Итоги
науки и тех. Серю Совр. Пробл. Мат. Фунд. направл. М. ВИНИТИ.
1990- 57, стр 5--177. (РЖМат, 1990).


\bibitem{ChelnokovCoPeriod}
  {\sl Г.~Р.~Челноков}  {\it О числе
запретов, задающих периодическую последовательность  }
Моделирование и анализ информационных систем, Т.13, N3 (2007)
66--70



\bibitem{Ab} A.Aberkane, Words whose complexity satisfies $\lim p(n)/n
= 1$ // Theor. Comp. Sci., 307, (2003), 31-46.


\bibitem{AHP} P.Ambro z, L. H\'akov \', E. Pelantov\'a, Properties of 3iet
preserving morphisms and their matrices // ProceedingsWORDS 2007,
Eds. P. Arnoux, N. Bedaride, J. Cassaigne, (2007), 18--24.


\bibitem{AR} P. Arnoux and G. Rauzy [1991],
Representation geometrique des suites the complexite $2n + 1$ //
Bull. Soc. Math. France 119, 199-215.

\bibitem{ArnoouxMaduitShTmrJ}
Arnoux, Pierre; Mauduit, Christian; Shiokawa, Iekata; Tamura,
Jun-ichi. {\it Complexity of sequences defined by billiard in the
cube}. Bull. Soc. Math. France 122 (1994), no. 1, 1--12.

\bibitem{Ba1} P.Bal\'azi, Infnite Words Coding Three-Interval
Exchange // diploma work CTU (2003).

\bibitem{Ba2} P.Bal\'azi, Substitution properties of ternary words coding 3-
interval exchange, // ProceedingsWORDS 2003, Eds. T. Harju and J.
Karhumдaki, (2003), 119-124.

\bibitem{BMP} P. Bal\'azi, Z. Mas\'akov\'a, E. Pelantov\'a, Characterization of
substitution invariant 3iet words, submitted to Integers //
arXiv:0709.2638, (2007).



\bibitem{BelovUzyRowen}
    {\sl  Kanel--Belov A., Rowen Louis H.; Vishne, Uzi}
   {\it Normal bases of $PI$-algebras.\/}
    Adv. in Appl. Math. 37 (2006), no. 3, 378--389.


\bibitem{C1} {\sl A.L.~Chernyat'ev}, {\it Balanced Words and Dynamical Systems} // Fundamental and Applied Mathematics,
2007, vol. 13, No 5, pp. 213--224

\bibitem{C2} {\sl A.L.~Chernyat'ev}, {\it Words with Minimal Growth Function} // Vestnik Mosk. Gos. Univ.,
2008.





\bibitem{BelovChernInt}
{\sl A.Ya.~Kanel-Belov, A.L.~Chernyat'ev.} {\it Describing the set
of words generated by interval exchange transformation}. Comm. in
Algebra, Vol. 38, No 7, July 2010, pages 2588--2605.



\bibitem{BFZ} V. Berth\'e, S. Ferenczi, L.Zamboni: Interactions between
dynamics, arithmetics and combina- torics: the good, the bad, and
the ugly, AMS Contemporary Math. 385 (2005), p. 333-364.

\bibitem{BC} {\sl A.Ya.~Belov and A.L.~Chernyat'ev}, {\it Describing Sturmian Words over an $n$-letter Alphabet} //
Math. Met. Appl. IV, MGSU, 1999, pp~122--128.

\bibitem{BS} J. Berstel, P. S\'{e}\'{e}bold, Sturmian words, in: M. Lothaire (Ed.) // Algebraic Combinatorics on Words,
Encyclopedia of Mathematics and Its Applications, Vol. 90,
Cambridge University Press, Cambridge, 2002 (Chap. 2).

\bibitem{B} J.Berstel, Resent results on Sturmian words //
Developments in language theory II, 13-24, World Scientific, 1996.


\bibitem{Cassaigne} J.Cassaigne. {\it Special factors with linear
subword complexity.} Developments in language theory, II
(Magdeburg, 1995), 25-34, World Sci. Publ., River Edge, NJ, 1996.


\bibitem{CassHubTroub} {\sl
Cassaigne, J. (F-CNRS-IML); Hubert, P. [Hubert, Pascal]
(F-CNRS-IML); Troubetzkoy, S. (F-CNRS-IML)} {\it Complexity and
growth for polygonal billiards.}  Ann. Inst. Fourier (Grenoble) 52
(2002), no. 3, 835--847.



\bibitem{DP} X. Droubay, G. Pirillo, Palindromes and sturmian
word// Theroret. Comput. Sci., 223:73-85, 1999.

\bibitem{DJP} X.Droubay, J.Justin, G.Pirillo, Episturmian words
and some construction of de Luca and Rauzy // Theoret. Comp.
Sci.,(2001) 539--553.



\bibitem{F} {\sl H. Furstenberg.} {\it Poincar\'e reccurence and number
theory.} Bull. Amer. Math. Soc., 5:211-234, 1981.

\bibitem{GR} R. L. Graham, Covering the Positive Integers by disjoints sets of the form $\{[n\alpha + \beta]: n=1, 2,
\ldots\}$ // J. Combin. Theory Ser A15 (1973) 354-358.

\bibitem{H} P. Hubert, Well balanced sequences // Theoret. Comput. Sci. 242 (2000) 91Ц108.

\bibitem{FHZ} S. Ferenczi, C. Holton, L. Zamboni, The structure of
three-interval exchange transformations II: a combinatorial
description of the trajectories// J. Anal. Math. 89 (2003), p.
239-276.

\bibitem{FZ} Ferenczi, L. Zamboni, A new induction for symmetric k-interval
exchange transformations and distances theorem, submitted,
http://iml.univ-mrs.fr/~ferenczi/fz1.pdf

\bibitem{FZ2} Ferenczi, L. Zamboni, Examples of $4$-interval exchange transformations, preprint (2006),
http://iml.univ-mrs.fr/~ferenczi/fz2.pdf

\bibitem{FZ3} Ferenczi, L. Zamboni, Languages of k-interval exchange
transformations, submitted,
http://iml.univ-mrs.fr/~ferenczi/fz3.pdf



\bibitem{AdL} A. de Luca, Sturmian words: structure,
combinatorics and their arithmetics // Theoret. Comp. Sci., 183,
(1997), 45-82.

\bibitem{DV1} A. de~Luca and S. Varricchio.
Combinatorial properties of uniformly recurrent words and an
application to semigroups. // {\it Inter. J. Algebra Comput.},
1(2):227--246, 1991.MR 92h:20084.

\bibitem{Kol} А.Т.Колотов, Апериодические последовательности и функции роста в алгебрах
// Алгебра и логика, 20 (1981), 138--154, 250.

\bibitem{LifshicDisser} {\it Application of Adic representations
in the investigations of metric, spectral and topological
properties of dynamical systems.}  Sanct-Petersburg, 1995, 176
pages.



\bibitem{L1} M. Lothaire, Combinatorics on Words // Encyclopedia of
Mathematics and its Applications, Addison-Wesley, Reading, MA,
1983, Vol. 17.

\bibitem{L2} M. Lothaire, Algebraic combinatorics on words.
A collective work by Jean Berstel, Dominique Perrin, Patrice
Seebold, Julien Cassaigne, Aldo De Luca, Steffano Varricchio,
Alain Lascoux, Bernard Leclerc, Jean-Yves Thibon, Veronique
Bruyere, Christiane Frougny, Filippo Mignosi, Antonio Restivo,
Christophe Reutenauer, Dominique Foata, Guo-Niu Han, Jacques
Desarmenien, Volker Diekert, Tero Harju, Juhani Karhumaki and
Wojciech Plandowski. With a preface by Berstel and Perrin.
Encyclopedia of Mathematics and its Applications, 90. Cambridge
University Press, Cambridge, 2002, 504 pp.


\bibitem{MauduitSubst}
{\sl Mauduit, C.} {\it Substitutions, arithmetic and finite
automata: an introduction. Substitutions in dynamics, arithmetics
and combinatorics,} 35-52, Lecture Notes in Math., 1794, Springer,
Berlin, 2002, 37B10 (11B85) PDF Clipboard Series Chapter

\bibitem{Milnor}
{\sl Milnor J.} {\it Problem 5603}. American Mathematical monthy,
1968, vol. 75, No 6, pages 675--686.

\bibitem{MorseHedlund}
M. Morse and G. A. Hedlund [1940], Symbolic dynamics II.
Sturmian trajectories, // Amer. J. Math. 62, 1-42.



\bibitem{Ra19}
G. Rauzy, {\it Nombres algebriques et substitutions}, Bull. Soc.
Math. France, 110 (1982), p. 147--178.



\bibitem{Ra} G. Rauzy,
Mots infnis en arithmetique, in: Automata on Infnite Words //
Ecole de Printemps d'Informatique Thfeorique, Le Mont Dore, May
1984, ed. M. Nivat and D. Perrin, Lecture Notes in Computer
Science, vol. 192, Springer-Verlag, Berlin etc., pp. 165-171, 1985

\bibitem{Ra2} G. Rauzy, Suites \'a termes dans un alphabet fini //
In {\it S\'emin. Th\'eorie des Nombres,} p. 25-01-25-16,
1982-1983, Expos\'e No 25.

\bibitem{Ra3} G. Rauzy, \'Echanges
d'intervalles et transformations induites, (in French), Acta
Arith. 34 (1979), p. 315-328.

\bibitem{RS} G. Rozenberg and A. Salomaa //
The Mathematical Theory of L Systems, Academic Press, New York
etc., 1980

\bibitem{R} G.Rote, Sequences with subword complexity $2n$ // J. Number Theory 46 (1994)
196--213.

\bibitem{Sal81}
Arto Salomaa.{\it Jewels of Formal Language Theory} // Computer
Science Press, 1981.

\bibitem{T1}R.Tijdeman, Decomposition of the integers as a
direct sum of two subsets // in: Number Theory, ed. by S. David,
Number Theory Seminar Paris 1992-93, Cambridge University Press,
(1995), 261-276

\bibitem{T2}R.Tijdeman, Fraenkel's conjecture for six
sequences // Discrete Mathematics, Volume 222,  Issue 1-3, 223 -
234, 2000

\bibitem{V} L. Vuillon, Balanced words // Bull. Belg. Math. Soc. Simon Stevin 10 (2003), no. 5, 787-805

\bibitem{V2} L.Vuillon, A characterization of Sturmian word by
return words // European Journal of Combinatorics (2001) 22,
263-275.

\bibitem{Vershik2}
Vershik, A. M. {\it
 The adic realizations of the ergodic actions with the homeomorphisms of
 the Markov compact and the ordered Bratteli diagrams.} Zap. Nauchn. Sem.
 S.-Peterburg. Otdel. Mat. Inst. Steklov. (POMI) 223 (1995), Teor.
 Predstav. Din. Sistemy, Kombin. i Algoritm. Metody. I, 120--126, 338;
 translation in J. Math. Sci. (New York) 87 (1997), no. 6, 4054--4058.


\bibitem {VL}
 Vershik, A. M.; Livshits, A. N.
{\it Adic models of ergodic transformations, spectral theory,
substitutions, and related topics. Representation theory and
dynamical systems}, 185--204, Adv. Soviet Math., 9, Amer. Math.
Soc., Providence, RI, 1992.


\bibitem{W} H.Weyl, \"{U}ber der gleichverteilung von zahlen mod
1,  Math. Ann., v. 77, 313-352, 1916.


\end{thebibliography}
\end{document}